\documentclass[12pt, leqno]{amsart}

\usepackage{graphicx}
\usepackage{amssymb,amsmath,amsthm,amscd}
\usepackage{mathrsfs}
\usepackage{enumerate}
\usepackage[usenames,dvipsnames]{color}
\usepackage[colorlinks=true, pdfstartview=FitV,
 linkcolor=blue,citecolor=blue,urlcolor=blue]{hyperref}
\usepackage[all]{xy}
\usepackage{verbatim}

\newenvironment{red}{\relax\color{red}}{\hspace*{.5ex}\relax}
\newcommand{\ber}{\begin{red}}
\newcommand{\er}{\end{red}}

\numberwithin{equation}{section}
\allowdisplaybreaks[4]

\newcommand{\nc}{\newcommand}
\nc{\op}{\operatorname}

\theoremstyle{plain}
\newtheorem{lemma}{Lemma}[subsection]
\newtheorem{prop}[lemma]{Proposition}
\newtheorem{theorem}[lemma]{Theorem}
\newcommand{\Prop}{\begin{prop}}
\newcommand{\enprop}{\end{prop}}
\newcommand{\Lemma}{\begin{lemma}}
\newcommand{\enlemma}{\end{lemma}}
\newcommand{\Th}{\begin{theorem}}
\newcommand{\enth}{\end{theorem}}
\newtheorem{corollary}[lemma]{Corollary}
\newcommand{\Cor}{\begin{corollary}}
\newcommand{\encor}{\end{corollary}}
\newtheorem{definition}[lemma]{Definition}
\newtheorem{conjecture}[lemma]{Conjecture}
\newcommand{\Def}{\begin{definition}}
\newcommand{\edf}{\end{definition}}
\newtheorem{sublemma}[lemma]{Sublemma}
\newcommand{\Sublemma}{\begin{sublemma}}
\newcommand{\ensub}{\end{sublemma}}

\theoremstyle{definition}
\newtheorem{remark}[lemma]{Remark}

\newtheorem{Convention}[lemma]{Convention}
\newcommand{\Conv}{\begin{Convention}}
\newcommand{\enconv}{\end{Convention}}
\newcommand{\Rem}{\begin{remark}}
\newcommand{\enrem}{\end{remark}}

\newcommand{\C}{\mathbb {C}}
\newcommand{\Q}{\mathbb {Q}}

\newcommand{\Z}{{\mathbb Z}}
\newcommand{\B}{{\mathbf{B}}}
\newcommand{\A}{{\mathbf A}}

\newcommand{\CC}{{\mathscr{C}}}
\newcommand{\one}{{\bf{1}}}
\newcommand{\seteq}{\mathbin{:=}}

\newcommand{\hd}{{\operatorname{hd}}}
\newcommand{\soc}{{\operatorname{soc}}}

\newcommand{\g}{{\mathfrak{g}}}

\newcommand{\Hom}{\operatorname{Hom}}

\newcommand{\isoto}[1][]{\mathop{\xrightarrow%
[{\raisebox{.3ex}[0ex][.3ex]{$\scriptstyle{#1}$}}]%
{{\raisebox{-.6ex}[0ex][-.6ex]{$\mspace{2mu}\sim\mspace{2mu}$}}}}}
\newcommand{\Ext}{\operatorname{Ext}}

\newcommand{\eq}{\begin{eqnarray}}
\newcommand{\eneq}{\end{eqnarray}}

\newcommand{\eqn}{\begin{eqnarray*}}
\newcommand{\eneqn}{\end{eqnarray*}}
\newcommand{\on}{\operatorname}
\newcommand{\Ker}{\on{Ker}}

\newcommand{\bni}{\be[{\rm(i)}]}
\newcommand{\bna}{\be[{\rm(a)}]}

\newcommand{\QED}{\end{proof}}
\newcommand{\Proof}{\begin{proof}}

\newcommand{\soplus}{\mathop{\mbox{\normalsize$\bigoplus$}}\limits}

\newcommand{\To}[1][{\hs{2ex}}]{\xrightarrow{\,#1\,}}

\newcommand{\ba}{\begin{array}}
\newcommand{\ea}{\end{array}}

\newcommand{\bi}{\begin{enumerate}[{\rm(i)}]}

\newcommand{\monoto}{\rightarrowtail}

\newcommand{\set}[2]{\left\{#1 \mathbin{;} #2 \right\}}

\newcommand{\hs}{\hspace*}

\newcommand{\eqsub}{\begin{subequations}\begin{eqnarray}}
\newcommand{\eneqsub}{\end{eqnarray}\end{subequations}}

\newcommand{\ol}{\overline}

\nc{\la}{\lambda}
\nc{\lam}{\lambda}
\nc{\U}[1][\g]{U_q(#1)}
\nc{\te}{\tilde{e}}
\nc{\tei}{\tilde{e}_i}
\nc{\tf}{\tilde{f}}
\nc{\tfi}{\tilde{f}_i}
\nc{\tU}{\widetilde U_q(\g)}
\nc{\tE}{\tilde{E}}
\nc{\tF}{\tilde{F}}

\nc{\tk}{\tilde{k}}
\nc{\tkone}{\tk_{\ol{1}}}
\nc{\teone}{\tilde{e}_{\ol{1}}}
\nc{\tfone}{\tilde{f}_{\ol{1}}}

\nc{\teibar}{\tilde{e}_{\ol{i}}} \nc{\tfibar}{\tilde{f}_{\ol{i}}}
\nc{\tki}{{\tk}_{\ol {i}}}

\nc{\BZ}{{\mathbb{Z}}}
\nc{\al}{\alpha}
\nc{\qs}{{q}}
\nc{\lan}{\langle}
\nc{\ran}{\rangle}
\nc{\re}{{\mathrm{re}}}
\nc{\wt}{\operatorname{wt}}
\nc{\ch}{\operatorname{ch}}
\nc{\Uf}[1][\g]{U^-_q(#1)}
\nc{\Ue}{U^+_q(\g)}
\nc{\eps}{\varepsilon}
\nc{\vphi}{\varphi}
\nc{\sphi}{\varphi^*}
\nc{\seps}{\varepsilon^*}

\nc{\nn}{\nonumber}
\def\max{{\mathop{\mathrm{max}}}}
\nc{\vp}{\varpi}
\nc{\cls}{{\operatorname{cl}}}
\nc{\Wt}{{\operatorname{Wt}}}
\nc{\Us}{U'_q(\g)}
\nc{\La}{\Lambda}
\nc{\ro}{{\rm(}}
\nc{\rf}{{\rm)}}
\nc{\norm}{{\mathrm{norm}}}
\nc{\qbox}{\quad\mbox}
\nc{\braid}{{\mathfrak{B}}}
\nc{\Ad}{\operatorname{Ad}}
\nc{\Aut}{\operatorname{Aut}}
\nc{\dt}[1]{\tilde{\tilde #1}}
\nc{\Sn}{S^{{\mathrm{norm}}}}
\nc{\aff}{{\rm{aff}}}
\nc{\rk}{{\mathrm{rk}}}
\nc{\tP}{\widetilde{P}}
\nc{\tW}{\widetilde{W}}
\nc{\Dyn}{\mathrm{Dyn}}
\nc{\tD}{\widetilde{\Delta}}
\nc{\height}{{\operatorname{ht}}}
\nc{\bl}{\bigl(}
\nc{\br}{\bigr)}
\nc{\Hecke}{\mathrm{H}}
\nc{\HA}{\Hecke^{\mathrm{A}}}
\nc{\HB}{\Hecke^{\mathrm{B}}}
\newcommand{\scbul}{{\,\raise1pt\hbox{$\scriptscriptstyle\bullet$}\,}}
\nc{\vac}{{\phi}}
\nc{\Bt}{\B_\theta(\g)}
\nc{\be}{\begin{enumerate}}
\nc{\ee}{\end{enumerate}}
\nc{\low}{{\mathrm{low}}}
\nc{\upper}{{\mathrm{up}}}
\nc{\Zodd}{\Z_{\mathrm{odd}}}
\nc{\Ft}[1][n]{\mathbb{P}\mathrm{ol}_{#1}}
\nc{\Ftf}[1][n]{\widetilde{\mathbb{P}\mathrm{ol}}_{#1}}
\nc{\KA}{\on{K}^{\mathrm{A}}}
\nc{\KB}{\on{K}^{\mathrm{B}}}
\nc{\Res}{\on{Res}}
\nc{\Fc}[1][{n,m}]{\mathbf{F}_{#1}}
\nc{\tphi}{\tilde{\varphi}}
\nc{\CO}{\mathscr{O}}
\nc{\inte}{\mathrm{int}}
\nc{\Oint}{\mathcal{O}^{\ge0}_{\inte}}
\nc{\vs}{\vspace}
\nc{\tL}{\widetilde{L}}
\nc{\tu}{\tilde{u}}
\nc{\noi}{\noindent}
\nc{\heigh}{\mathfrak{t}}
\nc{\lowest}{\mathfrak{l}}
\nc{\rootl}{\mathsf{Q}}
\nc{\cl}{{\rm{cl}}}
\nc{\uqpg}{U'_q(\mathfrak g)}
\nc{\Oh}{\widehat{\mathcal{O}}}

\newenvironment{rouge}
{\color{red}}
{}

\nc{\bred}{\begin{rouge}}
\nc{\ered}{\end{rouge}}
\nc{\KLR}{quiver Hecke algebra}
\nc{\KLRs}{quiver Hecke algebras}
\nc{\cor}{\mathbf{k}}
\nc{\cora}{{\cor(A)}}
\nc{\haut}{\mathrm{ht}}
\nc{\tens}{\mathop\otimes}
\nc{\gmod}{\mbox{-$\mathrm{gmod}$}}
\nc{\proj}{\mbox{-$\mathrm{proj}$}}
\nc{\gproj}{\mbox{-$\mathrm{gproj}$}}
\nc{\smod}{\mbox{-$\mathrm{mod}$}}

\nc{\h}{\mathfrak h}
\nc{\Rnorm}{R^{\rm{norm}}}
\nc{\Runiv}{R^{\rm{univ}}}
\nc{\Vhat}{\widehat{V}}
\nc{\F}{\mathcal{F}}

\def\T{{\mathcal T}}

\nc{\fd}[1][A]{\on{\mathrm{flat.dim}_{#1}}}
\nc{\bP}{{\mathbb{P}}}
\nc{\bPh}{\widehat{\mathbb{P}}}
\nc{\bK}{\widehat{\mathbb{K}}}
\nc{\bV}[1][{n}]{\widehat{V}^{\otimes{#1}}}
\nc{\bVK}[1][{n}]{\widehat{V}^{\otimes{#1}}_K}
\nc{\opp}{\mathrm{opp}}
\nc{\col}{\colon}
\nc{\bnum}{\be[{\rm(i)}]}
\nc{\oep}{\epsilon}
\nc{\qtext}{\quad\text}
\nc{\qtextq}[1]{\quad\text{#1}\quad}
\nc{\longtwoheadrightarrow}[1][]{\xymatrix{\ar@{->>}[r]^-{{#1}}&}}
\nc{\epiTo}[1][]{\longtwoheadrightarrow[{#1}]}
\nc{\epito}{\twoheadrightarrow}
\nc{\monoTo}[1][]{\xymatrix{\ar@{>->}[r]^-{{#1}}&}}
\nc{\sym}{\mathfrak{S}}
\nc{\inp}[1]{{({#1})_{\mathrm{n}}}}
\nc{\rtl}{\rootl}
\nc{\wtd}{\widetilde}
\nc{\etens}{\boxtimes}
\nc{\ds}[1]{\mathrm{d}(#1)}
\nc{\rmat}[1]{{r}_{\mspace{-2mu}\raisebox{-.5ex}{${\scriptstyle{#1}}$}}}
\nc{\shc}{\mathcal{C}}
\nc{\Fct}{{\on{Fct}}}
\nc{\tC}{\widetilde{\shc}}
\nc{\Zp}{\Z_{\ge0}}
\nc{\tPhi}{\widetilde{\Phi}}
\nc{\tT}{{\tilde{\T}}}
\nc{\Ob}{\on{Ob}}
\nc{\bwr}{\mbox{\large$\wr$}}
\nc{\Img}{\on{Im}}
\nc{\Ab}{\mathcal{A}^{\mathrm{big}}}
\nc{\Sb}{\mathcal{S}^{\mathrm{big}}}
\nc{\As}{\mathcal{A}}
\nc{\Ss}{\mathcal{S}}
\nc{\ntens}{\widetilde{\otimes}}
\nc{\hR}{\widehat{R}}
\nc{\nconv}{\star}
\nc{\ts}{\tilde{s}}
\nc{\sho}{\mathcal{O}}
\nc{\bc}{\begin{cases}}
\nc{\ec}{\end{cases}}
\nc{\UA}{U_q'(\widehat{\mathfrak{sl}_N})}
\nc{\KR}{R_K}
\nc{\cQ}{\mathcal{Q}}
\nc{\Irr}{\mathcal{I}rr}
\nc{\tQ}{\widetilde{\cQ}}
\nc{\bs}{\mathbf{s}}
\nc{\bL}{\mathbb{L}}
\nc{\KP}{{\mathrm{KP}}}
\nc{\db}{\mathsf{b}^*}
\nc{\bfa}{{\mathbf{a}}}
\nc{\bfc}{{\mathbf{c}}}
\nc{\Po}{(P_\cl)_0}
\renewcommand{\Im}{\op{Im}}
\nc{\mono}{\rightarrowtail}
\nc{\tr}{\on{tr}}
\nc{\K}{\on{K}}

\newlength{\mylength}
\title[Symmetric quiver Hecke algebras and R-matrices II]
{Symmetric quiver Hecke algebras and R-matrices of quantum affine algebras II}

\author[S.-J. Kang, M. Kashiwara, M. Kim]{Seok-Jin Kang$^{1}$, Masaki Kashiwara$^{2}$ and Myungho Kim}

\address{Department of Mathematical Sciences
         and
         Research Institute of Mathematics \\
         Seoul National University \\ Seoul 151-747, Korea}

         \email{sjkang@math.snu.ac.kr}

\address{Research Institute for Mathematical Sciences \\
          Kyoto University \\ Kyoto 606-8502, Japan \\
          \& Department of Mathematical Sciences
         and
         Research Institute of Mathematics \\
         Seoul National University \\ Seoul 151-747, Korea}

         \email{masaki@kurims.kyoto-u.ac.jp}

\address{School of Mathematics, Korea Institute for Advanced Study \\ Seoul 130-722, Korea}
         \email{mhkim@kias.re.kr}

\thanks{$^{1}$This work was supported by NRF Grant \# 2013-035155 and by NRF Grant \# 2012-0008829}
\thanks{$^{2}$This work was partially supported by Grant-in-Aid for
Scientific Research (B) 22340005, Japan Society for the Promotion of
Science.}
\keywords{Quantum affine algebra, Quiver Hecke algebra, Quantum group}

\subjclass[2010]
{81R50, 16G, 16T25,17B37}

\date{August 3, 2013}

\begin{document}

\begin{abstract}
Let $\g$ be an untwisted affine Kac-Moody algebra
of type $A^{(1)}_n$ $(n \ge 1)$ or $D^{(1)}_n$ $(n \ge 4)$ and
let  $\g_0$ be the underlying finite-dimensional simple Lie subalgebra of $\g$. For each Dynkin quiver $Q$ of type $\g_0$,
Hernandez and Leclerc (\cite{HL11}) introduced a
tensor subcategory  $\CC_Q$ of the category of finite-dimensional integrable $\uqpg$-modules and proved that the Grothendieck ring of $\CC_Q$ is isomorphic to  $\C [N]$, the coordinate ring of the unipotent group $N$  associated with $\g_0$.
 We apply the generalized quantum affine Schur-Weyl duality
introduced in \cite{KKK13}
to construct an exact functor $\F$
from the category of finite-dimensional graded $R$-modules
to the category $\CC_Q$, where  $R$ denotes the symmetric quiver Hecke algebra
associated to $\g_0$.
We prove that the  homomorphism induced  by the functor $\F$ coincides with the homomorphism of Hernandez and Leclerc and show that the functor $\F$ sends the
simple modules to the simple  modules.
\end{abstract}

\maketitle

\section*{Introduction}
Recently, Khovanov-Lauda (\cite{KL09}) and Rouquier (\cite{R08}) independently introduced the
 {\it quiver Hecke algebras} which categorify the negative half of quantum groups.
More precisely, if $U_q(\g)$ is a quantum group associated with a symmetrizable Kac-Moody algebra $\g$, then  there exists a family of graded algebras $\{R(n) \}_{n \in \Z_{\ge 0}}$ such that
the Grothendieck ring of the category consisting of finite-dimensional graded $R(n)$-modules is isomorphic to  the integral form $U^-_{\Z[q,q^{-1}]}(\g)^\vee$
of the dual of negative half of $U_q(\g)$.
Soon after, it is also shown  in \cite{VV09, R11}  that  if $\g$ is symmetric, then the isomorphism classes of finite-dimensional self-dual simple $R(n)$-modules correspond to the upper global basis (=Lusztig's dual canonical basis).
In particular, if $\g$  is a simply-laced  finite type semisimple Lie algebra and if we forget the grading of $R(n)$, then quiver Hecke algebras provide a categorification of the coordinate ring $\C[N]$ of the unipotent group associated with  $\g$, since  $U^-_{\Z[q,q^{-1}]}(\g)^\vee$  is specialized to  $\C[N]$ at  $q=1$.

On the other hand, in \cite{HL11} Hernandez and Leclerc  proposed  another categorification of $\C[N]$  for a simply-laced finite-dimensional simple Lie algebra
$\g$  by using representations of quantum affine algebras.
Let us briefly recall their construction.
Let $\g$ be an untwisted affine  Lie  algebra  associated with
the finite-dimensional  semisimple Lie algebras $\g_0$ which is simply-laced.
Let $\uqpg$ be the quantum affine algebra corresponding to $\g$ and let $\CC_\g$ be the category of finite-dimensional integrable $\uqpg$-modules.
For each  Dynkin quiver  $Q$ of type $\g_0$, Hernandez and Leclerc defined an abelian subcategory $\CC_Q$ of $\CC_\g$ which contains some evaluation modules of fundamental representations parameterized by the vertices of Auslander-Reiten quiver $\Gamma_Q$ of $Q$.
The category $\CC_Q$ is stable under taking tensor product so that the  Grothendieck group $ \K(\CC_Q)$ is endowed with  a ring structure.
They showed that the complexified Grothendieck ring $\C \otimes _\Z \K(\CC_Q)$ is isomorphic to the coordinate ring $\C [N]$
 of the maximal unipotent group $N$
associated with $\g_0$.
Moreover, under  this  isomorphism,
the set of  isomorphism classes
of simple objects in $\CC_Q$ corresponds to the upper global basis of $\C[N]$.

To summarize, we have two different categorifications of the algebra $\C[N]$ and its upper global basis;  one  via quiver Hecke algebras and
 the other  via quantum affine algebras.
 Hence it is natural to ask whether they are related or not.
 More precisely, Hernandez and Leclerc asked whether there is a tensor functor from the category consisting of finite-dimensional graded $R(n)$-modules to the category $\CC_Q$ which sends a simple module to a simple module (see \cite[\S 1.6]{HL11}).

 In this paper, we answer the above question affirmatively.
The desired functor is an example of the {\it generalized quantum affine Schur-Weyl duality functors}, which are  developed in \cite{KKK13}.
Suppose that we have a set $J$ of pairs of good $\uqpg$-modules and invertible elements in the base field.
It is equivalent to saying that we have a set of  evaluation modules of some good $\uqpg$-modules.
One can  associate a quiver $\Gamma^J$ with $J$, and hence a symmetric quiver Hecke algebra $R^J$, by investigating the  pole distribution of the normalized R-matrices between the evaluation modules.
Then we have a tensor functor $\F_m : R^J(m) \gmod \to \CC_\g$, which enjoys several  nice properties.
For example, if the underlying graph of $\Gamma^J$ is of finite simply-laced type, then the functor $\F_m$ is exact.

To apply this general scheme to the above question, we first need to know the denominators of normalized R-matrices between fundamental representations.
While  all the denominators are already known for the type $A^{(1)}_n$ (see, for example \cite{DO94}),
only  partial cases were known for the type $D^{(1)}_n$.
So we calculate all the unknown denominators for type $D^{(1)}_n$ in the appendix. For the type $E^{(1)}_n$, we provide a conjecture on the order of poles of normalized R-matrices (Conjecture \ref{conj:simple poles}).
Next, we take a set $J$ consisting of some evaluation modules of fundamental representations in $\CC_Q$, which correspond to the simple representations of the quiver $Q$ in the Auslander-Reiten quiver $\Gamma_Q$ (see \eqref{eq:J}).
Then by the information of the poles and their order of normalized R-matrices, we conclude that the associated quiver $\Gamma^J$ is isomorphic to the opposite quiver $Q^{\rm rev}$  and hence there exists a tensor functor $\F \col R^J \gmod \to \CC_Q$ (Theorem \ref{thm:corresponding KLR}). Since the underlying graph of $\Gamma^J$ is of type $\g_0$, we have $ U^-_\A(\g_0)^\vee \simeq \K(R^J \gmod)$ and the functor $\F$ is exact.
We show that the homomorphism $\phi_\F\col  \K(R^J \gmod) \to \K(\CC_Q)$ induced by $\F$ is
the same as the isomorphism in \cite{HL11} if we forget the grading on the category $R^J \gmod$ (Theorem \ref{thm:dual pbw generators}). As a corollary, we conclude that the functor sends a simple module to a simple module (Corollary \ref{cor:simple to simple}).
 We note that our proofs of these results rely on the results of
Hernandez-Leclerc (\cite{HL11}).

This paper is organized as follows:
In Section 1, we recall the basic materials on quiver Heck algebras and quantum affine algebras.
In Section 2, we recall the construction of generalized quantum affine Schur-Weyl duality functors following \cite{KKK13}.
In Section 3, the category $\CC_Q$ is reviewed and several lemmas for the next section are proved.
In Section 4, we  define  the functor $\F$ and establish our main theorems.
The appendix is devoted to the calculation of  the  denominators of
 the normalized R-matrices between fundamental representations for the quantum affine algebra $U'_q(D^{(1)}_n)$.

\smallskip

{\bf Acknowledgements}\
We would like to thank Bernard Leclerc
who kindly explained us his results with David Hernandez.
We would like to thank also Se-jin Oh for many
helpful discussions.

\section{Quiver Hecke algebras and  quantum affine algebras}

\subsection{Quantum groups and global basis}\label{subsec:Qgroups}
In this section, we recall the definition of quantum groups and global basis.
Let $I$
be an index set. A \emph{Cartan datum} is a quintuple $(A,P,
\Pi,P^{\vee},\Pi^{\vee})$ consisting of
\begin{enumerate}[(a)]
\item an integer-valued matrix $A=(a_{ij})_{i,j \in I}$,
called  a \emph{symmetrizable generalized Cartan matrix},
 which satisfies
\bni
\item $a_{ii} = 2$ $(i \in I)$,
\item $a_{ij} \le 0 $ $(i \neq j)$,
\item $a_{ij}=0$ if $a_{ji}=0$ $(i,j \in I)$,
\item there exists a diagonal matrix
$D=\text{diag} (\mathsf s_i \mid i \in I)$ such that $DA$ is
symmetric, and $\mathsf s_i$ are positive integers.
\end{enumerate}

\item  a free abelian group $P$,
called the \emph{weight lattice},
\item $\Pi= \{ \alpha_i \in P \mid \ i \in I \}$, called
the set of \emph{simple roots},
\item $P^{\vee}\seteq\Hom(P, \Z)$, called the \emph{co-weight lattice},
\item $\Pi^{\vee}= \{ h_i \ | \ i \in I \}\subset P^{\vee}$, called
the set of \emph{simple coroots},
\end{enumerate}

satisfying the following properties:
\begin{enumerate}
\item[(i)] $\langle h_i,\alpha_j \rangle = a_{ij}$ for all $i,j \in I$,
\item[(ii)] $\Pi$ is linearly independent,
\item[(iii)] for each $i \in I$, there exists $\Lambda_i \in P$ such that
           $\langle h_j, \Lambda_i \rangle =\delta_{ij}$ for all $j \in I$.
\end{enumerate}
We call $\Lambda_i$ a \emph{fundamental weight}.
The free abelian group $\rootl\seteq\soplus_{i \in I} \Z \alpha_i$ is called the
\emph{root lattice}. Set $\rootl^{+}= \sum_{i \in I} \Z_{\ge 0}
\alpha_i\subset\rootl$ and $\rootl^{-}= \sum_{i \in I} \Z_{\le0}
\alpha_i\subset\rootl$. For $\beta=\sum_{i\in I}m_i\al_i\in\rootl$,
we set
$|\beta|=\sum_{i\in I}|m_i|$.

Set $\mathfrak{h}=\Q \otimes_\Z P^{\vee}$.
Then there exists a symmetric bilinear form $(\quad , \quad)$ on
$\mathfrak{h}^*$ satisfying
$$ (\alpha_i , \alpha_j) =\mathsf s_i a_{ij} \quad (i,j \in I)
\quad\text{and $\lan h_i,\lambda\ran=
\dfrac{2(\alpha_i,\lambda)}{(\alpha_i,\alpha_i)}$ for any $\lambda\in\mathfrak{h}^*$ and $i \in I$}.$$

Let $q$ be an indeterminate. For each $i \in I$, set $q_i = q^{\,\mathsf s_i}$.

\begin{definition} \label{def:qgroup}
The {\em quantum group} $U_q(\g)$ associated with a Cartan datum
$(A,P,\Pi,P^{\vee}, \Pi^{\vee})$ is the  algebra  over
$\mathbb Q(q)$ generated by $e_i,f_i$ $(i \in I)$ and
$q^{h}$ $(h \in P^\vee)$ satisfying  the  following relations:
\begin{equation*}
\begin{aligned}
& q^0=1,\ q^{h} q^{h'}=q^{h+h'} \ \ \text{for} \ h,h' \in P,\\
& q^{h}e_i q^{-h}= q^{\lan h, \alpha_i\ran} e_i, \ \
          \ q^{h}f_i q^{-h} = q^{-\lan h, \alpha_i\ran} f_i \ \ \text{for} \ h \in P^\vee, i \in
          I, \\
& e_if_j - f_je_i = \delta_{ij} \dfrac{K_i -K^{-1}_i}{q_i- q^{-1}_i
}, \ \ \mbox{ where } K_i=q^{\mathsf s_i h_i}, \\
& \sum^{1-a_{ij}}_{r=0} (-1)^r \left[\begin{matrix}1-a_{ij}
\\ r\\ \end{matrix} \right]_i e^{1-a_{ij}-r}_i
         e_j e^{r}_i =0 \quad \text{ if } i \ne j, \\
& \sum^{1-a_{ij}}_{r=0} (-1)^r \left[\begin{matrix}1-a_{ij}
\\ r\\ \end{matrix} \right]_i f^{1-a_{ij}-r}_if_j
        f^{r}_i=0 \quad \text{ if } i \ne j.
\end{aligned}
\end{equation*}
\end{definition}

Here, we set $[n]_i =\dfrac{ q^n_{i} - q^{-n}_{i} }{ q_{i} - q^{-1}_{i} },\quad
  [n]_i! = \prod^{n}_{k=1} [k]_i$ and
  $\left[\begin{matrix}m \\ n\\ \end{matrix} \right]_i= \dfrac{ [m]_i! }{[m-n]_i! [n]_i! }\;$
  for each $m,n \in Z_{\ge 0}$, $i \in I$.

We have a comultiplication
$\Delta \colon U_q(\g) \rightarrow U_q(\g) \otimes U_q(\g) $
given by
 \begin{align*}
 & \Delta (q^h) = q^h \otimes q^h, &
 & \Delta (e_i) = e_i \otimes K_i^{-1} + 1 \otimes e_i, &
 & \Delta (f_i) = f_i \otimes 1 + K_i \otimes f_i.
  \end{align*}

Let $U_q^{+}(\g)$ (resp.\ $U_q^{-}(\g)$) be the subalgebra of
$U_q(\g)$ generated by  the $e_i$'s (resp.\  the $f_i$'s), and let $U^0_q(\g)$
be the subalgebra of $U_q(\g)$ generated by $q^{h}$ $(h \in
P^{\vee})$. Then we have the \emph{triangular decomposition}
$$ U_q(\g) \simeq U^{-}_q(\g) \otimes U^{0}_q(\g) \otimes U^{+}_q(\g),$$
and the {\em weight space decomposition}
$$U_q(\g) = \bigoplus_{\beta \in \rootl} U_q(\g)_{\beta},$$
where $U_q(\g)_{\beta}\seteq\set{ x \in U_q(\g)}{\text{$q^{h}x q^{-h}
=q^{\langle h, \beta \rangle}x$ for any $h \in P$}}$.

We have the following automorphisms on $U_q(\g)$:

\bna
\item $\Q$-algebra involution $- \col U_q(\g) \rightarrow U_q(\g)$ given by
 \eqn \overline{e_i} ={e_i}, \ \overline {f_i}=f_i, \ \overline{q^h}=q^{-h},  \ \text{and} \  \overline q =q^{-1}.
 \eneqn
\item $\Q(q)$-algebra involution $\vee \col U_q(\g) \rightarrow U_q(\g)$ given by
 \eqn e_i^{\vee}=f_i, \ f_i^{\vee}=e_i, \ \text{and} \ (q^h)^{\vee}=q^{-h}.
 \eneqn
\item $\Q(q)$-algebra anti-involution $* \col U_q(\g) \rightarrow U_q(\g)$ given by
\eqn  e^*_i=e_i, \ f^*_i=f_i, \ \text{and} \ (q^h)^*=q^{-h}.
 \eneqn
 \ee

Let $\A= \Z[q, q^{-1}]$ and set
$$e_i^{(n)} = e_i^n / [n]_i!, \quad f_i^{(n)} =
f_i^n / [n]_i! \ \ (n \in \Z_{\ge 0}).$$
Let $U_{\A}^{+}(\g)$ (resp.\ $U_{\A}^{-}(\g)$) be the
$\A$-subalgebra of $U_q(\g)$ generated by $e_i^{(n)}$ (resp.\
$f_i^{(n)}$) for $i\in I$, $n \in \Z_{\ge 0}$.

For any $i \in I$,
there exists a unique $\Q(q)$-linear endomorphisms $e'_i$
of $U^-_q(\g)$  such that
\eqn
e'_i(f_j)=\delta_{i,j} \ (j \in I), &
 e'_i(xy) = e'_i (x) y + q_i^{\langle h_i, \beta \rangle} x e'_i(y) \ (x \in U^{-}_q(\g)_{\beta}, y \in U^-_q(\g)).
\eneqn

For each $ i \in  I$, any element $x \in U_q^-(\g)$ can be written uniquely as
\eqn
x = \sum_{n \ge 0} f_i^{(n)} x_n\quad\text{ with $x_n \in \Ker(e'_i)$.}
\eneqn
We define the {\it Kashiwara operators} $\tilde e_i, \tilde f_i$ on
$U_q^-(\g)$ by
\eqn
\tilde e_i x = \sum_{n \ge 1} f_i^{(n-1)}x_n, & \tilde f_i x = \displaystyle\sum_{n \ge 0} f_i^{(n+1)}x_n,
\eneqn
and set
\eqn
&L(\infty) = \sum_{\ell \ge0, i_i, \ldots, i_\ell \in I}
\A_0 \tilde f_{i_1} \cdots \tilde f_{i_\ell} \cdot \one \subset U_q^-(\g),
\quad \overline{L(\infty)} = \set{\overline x}{ x \in L(\infty)}, \\
&B(\infty) = \set{\tilde f_{i_1} \cdots \tilde f_{i_\ell} \cdot \one \mod q L(\infty)}{\ell \ge0, i_i, \ldots, i_\ell \in I}  \subset L(\infty) / q L(\infty),
\eneqn
where $\A_0=\set{g \in \Q(q)}{g  \ \text{is regular at} \ q=0}$.

The set $B(\infty)$ is a $\Q$-basis of $L(\infty) / q L(\infty)$ and the natural map
\eqn L(\infty) \cap \overline{L(\infty)} \cap U^-_\A(\g) \rightarrow L(\infty) / q L(\infty)\eneqn
is a $\Q$-linear isomorphism.
Let us denote the inverse of the above isomorphism by $G^{\low}$.
Then the set
$$\B^\low \seteq  \set{G^\low(b)}{b \in B(\infty)}$$
forms an $\A$-basis of $U_\A^-(\g)$ and is called by the {\it lower global basis} of $U_q^-(\g)$.

Recall that there exists a unique non-degenerate symmetric bilinear form $(\ ,\ )$ on $U^-_q(\g)$ such that
\eqn (\one,\one )=1, & (f_i u, v)=(u,e'_i v) \quad
\text{for $i \in I$, $u,v  \in U_q^-(\g)$}
\eneqn
(see \cite[\S 3.4]{Kas91}).
Set $$U^-_\A(\g)^\vee \seteq \set{x \in U^-_q(\g)}
{ \text{$(x,y) \in \A$ for all $y \in U^-_\A(\g)$} }.$$
Then $U^-_\A(\g)^\vee$ has an $\A$-algebra structure as a subalgebra of $U^-_q(\g)$.
For each $b \in B(\infty)$, define $G^{\upper}(b) \in U^-_q(\g)$ as the element satisfying
$$(G^{\upper}(b), G^{\low}(b')) =\delta_{b,b'} $$ for $b' \in B(\infty)$.
The set
$$\B^\upper \seteq  \set{G^\upper(b)}{b \in B(\infty)}$$
forms an $\A$-basis of $U_\A^-(\g)^\vee$ and is called the {\it upper global basis} of $U_q^-(\g)$.
Note that we have $*(\B^\low) = \B^\low$, and $*(\B^\upper) = \B^\upper$ (\cite{Kas91, Kas932}).

Define the {\it dual bar involution} $\db\col U_q^-(\g) \to U_q^-(\g)$ so that
$\db(x)$ satisfies
$$(\db(x),y) = \overline{(x,\overline y)}$$
for all $ y \in U_q^-(\g)$ (see \cite[\S~3.1]{Kimura12}).
By the  definition, we have $\db(G^\upper(b))  = G^\upper(b)$ for all $b \in B(\infty)$.

\subsection{Quiver Hecke algebras}
In this section,
we recall the definition of the \KLRs\ introduced in \cite{KL09, R08}.
Let $\cor$ be an arbitrary base field.
Let $J$ be a finite index set and
$A=(a_{i,j})_{i,j \in J}$ a symmetrizable generalized Cartan matrix as in the preceding subsection.
Let $\rtl^+=\soplus\nolimits_{i \in J} Z_{\ge 0} \al_i$ be the positive root lattice of the symmetrizable Kac-Moody Lie algebra corresponding to $A$.

For $i,j\in J$ such that $i\not=j$, set
$$S_{i,j}=\set{(p,q)\in\Z_{\ge0}^2}{(\al_i , \al_i)p+(\al_j , \al_j)q=-2(\al_i , \al_j)}.
$$
Let us take  a family of polynomials $(Q_{ij})_{i,j\in J}$ in $\cor[u,v]$
which are of the form
\begin{equation} \label{eq:Q}
Q_{ij}(u,v) = \begin{cases}\hs{5ex} 0 \ \ & \text{if $i=j$,} \\
\sum\limits_{(p,q)\in S_{i,j}}
t_{i,j;p,q} u^p v^q\quad& \text{if $i \neq j$}
\end{cases}
\end{equation}
with $t_{i,j;p,q}\in\cor$. We assume that
they satisfy $t_{i,j;p,q}=t_{j,i;q,p}$ (equivalently, $Q_{i,j}(u,v)=Q_{j,i}(v,u)$) and
$t_{i,j:-a_{ij},0} \in \cor^{\times}$.

We denote by
$\sym_{n} = \langle s_1, \ldots, s_{n-1} \rangle$ the symmetric group
on $n$ letters, where $s_i\seteq (i, i+1)$ is the transposition of $i$ and $i+1$.
Then $\sym_n$ acts on $J^n$ by place permutations.

\begin{definition}
The {\em quiver Hecke algebra  $R(n)$ of degree $n$
associated with the generalized Cartan matrix $A$ and the matrix $(Q_{ij})_{i,j \in J}$} is the associative algebra
over $\cor$ generated by the elements $\{ e(\nu) \}_{\nu \in  J^n }$, $ \{x_k \}_{1 \le k
\le n}$, $\{ \tau_m \}_{1 \le m \le n-1}$ satisfying the following
defining relations:
\eqn
\begin{aligned}
& e(\nu) e(\nu') = \delta_{\nu, \nu'} e(\nu), \ \
\sum_{\nu \in  J^n } e(\nu) = 1, \\
& x_{k} x_{m} = x_{m} x_{k}, \ \ x_{k} e(\nu) = e(\nu) x_{k}, \\
& \tau_{m} e(\nu) = e(s_{m}(\nu)) \tau_{m}, \ \ \tau_{k} \tau_{m} =
\tau_{m} \tau_{k} \ \ \text{if} \ |k-m|>1, \\
& \tau_{k}^2 e(\nu) = Q_{\nu_{k}, \nu_{k+1}} (x_{k}, x_{k+1})
e(\nu), \\
& (\tau_{k} x_{m} - x_{s_k(m)} \tau_{k}) e(\nu) = \begin{cases}
-e(\nu) \ \ & \text{if} \ m=k, \nu_{k} = \nu_{k+1}, \\
e(\nu) \ \ & \text{if} \ m=k+1, \nu_{k}=\nu_{k+1}, \\
0 \ \ & \text{otherwise},
\end{cases} \\
& (\tau_{k+1} \tau_{k} \tau_{k+1}-\tau_{k} \tau_{k+1} \tau_{k}) e(\nu)\\
& =\begin{cases} \dfrac{Q_{\nu_{k}, \nu_{k+1}}(x_{k},
x_{k+1}) - Q_{\nu_{k}, \nu_{k+1}}(x_{k+2}, x_{k+1})} {x_{k} -
x_{k+2}}e(\nu) \ \ & \text{if} \
\nu_{k} = \nu_{k+2}, \\
0 \ \ & \text{otherwise}.
\end{cases}
\end{aligned}
\eneqn
\end{definition}

The above relations are homogeneous  by assigning
\begin{equation*} \label{eq:Z-grading}
\deg e(\nu) =0, \quad \deg\; x_{k} e(\nu) = (\alpha_{\nu_k}
, \alpha_{\nu_k}), \quad\deg\; \tau_{l} e(\nu) = -
(\alpha_{\nu_l} , \alpha_{\nu_{l+1}}),
\end{equation*}
and hence $R(n)$ has a $\Z$-graded algebra structure.

There is an algebra involution $\psi$ on $R(n)$ given by
\eq \label{eq:auto psi}
e(\nu) \mapsto  e(\overline{\nu}), \quad
x_k \mapsto   x_{n-k+1}, \quad
\tau_{ l} e(\nu) \mapsto
(-1)^{\delta(\nu_l=\nu_{l+1})} \tau_{n- l} e(\overline{\nu}),
\eneq
where $\overline{\nu}$ denotes the reversed sequence $(\nu_n, \nu_{n-1}, \ldots, \nu_2, \nu_1)$.

For $n \in \Z_{\ge 0}$ and $\beta \in \rtl^+$ such that $|\beta| = n$, we set
\eqn
J^{\beta} = \set{\nu = (\nu_1, \ldots, \nu_n) \in J^{n}}
{ \alpha_{\nu_1} + \cdots + \alpha_{\nu_n} = \beta }, \quad e(\beta) =  \sum_{\nu \in J^\beta} e(\nu).
\eneqn
Then $e(\beta)$ is a central idempotent in $R(n)$ and the algebra
$R(\beta) \seteq e(\beta)R(n)e(\beta)$
is called the \emph{\KLR\ at $\beta$}.

 For a graded $R(n)$-module $M=\soplus_{k \in \Z} M_k$, we define
$qM =\soplus_{k \in \Z} (qM)_k$, where
 \begin{align*}
 (qM)_k = M_{k-1} & \ (k \in \Z).
 \end{align*}
We call $q$ the \emph{grading shift functor} on the category of graded $R(n)$-modules.

For $m,n \in \Z_{\ge 0}$, let
\eqn
R(m) \otimes R(n) \to R(m+n)
\eneqn
be the $\cor$-algebra homomorphism given by
$e(\mu) \otimes e(\nu) \mapsto e(\mu * \nu) \ (\mu \in J^m, \nu \in J^n)$,
$x_k\otimes 1 \mapsto x_k \ (1 \le k \le m)$,
$1\otimes x_k \mapsto x_{m+k} \ (1 \le k \le n)$
$\tau_k\otimes 1 \mapsto \tau_k \ (1 \le k <m)$,
$1\otimes \tau_k \mapsto \tau_{m+k} \ (1 \le k <n)$.
Here $\mu * \nu $ is the concatenation of $\mu $ and $\nu$; i.e., $\mu * \nu=(\mu_1,\ldots,\mu_m, \nu_1,\ldots,\nu_n)$.
For an $R(m)$-module $M$ and an $R(n)$-module $N$, the $R(m+n)$-module
\eqn
M \circ N \seteq R(m+n) \otimes_{R(m) \otimes R(n)} \bl M \otimes N\br
\eneqn
is called the \emph{convolution product of $M$ and $N$ }.
 Let us denote by $R(n)\gmod$ the category of finite-dimensional
graded $R(n)$ module, and set
$$R\gmod\seteq \soplus_{n\ge0}R(n)\gmod.$$

It is known that  \KLRs\ \emph{categorify} the negative half of the corresponding quantum group and global bases. More precisely, we have the following theorem.

\begin{theorem}[{\cite{KL09, R08}, \cite{VV09, R11}}] \label{Thm:categorification}
 For a given symmetrizable generalized  Cartan matrix $A$,
let us take a parameter matrix $(Q_{ij})_{i,j \in J}$
satisfying the conditions in \eqref{eq:Q}.
Consider the corresponding quantum group  $U_q(\mathfrak g)$  and the corresponding \KLR\ $R(n)$.
Then there exists an $\A$-algebra isomorphism
\begin{align}
  U^-_\A(\mathfrak g)^{\vee} \isoto \K(R \gmod),
\end{align}
where the multiplication and the $\A$-action on $\K(R\gmod)$ are given
by the convolution product and the grading shift functor, respectively.

Assume further that $A$ is symmetric and   $Q_{i,j}(u,v)$ is a polynomial in $u-v$.
Then under the above isomorphism, the upper global basis corresponds to the set of isomorphism classes  of self-dual simple modules in $\K(R\gmod)$.
\end{theorem}

\subsection{Quantum affine algebras and the category $\CC_\g$} \label{subec:quantum affine}
In this section, we will review the quantum affine algebras and their finite-dimensional integrable modules.
Hereafter, we take the algebraic closure of $\C(q)$
in $\cup_{m >0}\C((q^{1/m}))$ as a base field $\cor$ for quantum affine algebras.

Let  $A=(a_{ij})_{i,j \in I}$ be a generalized
Cartan matrix of affine type; i.e., $A$ is positive semi-definite of corank 1.
 We choose $0\in I$ as the leftmost vertices in the tables
in \cite[{pages 54, 55}]{Kac} except $A^{(2)}_{2n}$-case  in which
we take the longest simple root as $\al_0$.
Set $I_0=I\setminus\{0\}$.
We take a Cartan datum $(A,P,\Pi,P^{\vee},\Pi^{\vee})$ as follows.

The weight lattice $P$ is given by
\begin{align*}
  P = \Bigl(\soplus_{i\in I}\Z \La_i\Bigr) \oplus \Z \delta
\end{align*}
and the simple roots are given by
$$\al_i=\sum_{j\in I}a_{ji}\La_j+\delta(i=0)\delta.$$
Also, the simple coroots $h_i\in P^\vee=\Hom_\Z (P,\Z)$ are given by
\eqn
&& \lan h_i,\La_j\ran=\delta_{ij},\quad \lan h_i,\delta\ran=0.
\eneqn

Let $\{c_i \}_{i\in I}$, $\{d_i \}_{i\in I}$ be families of relatively prime positive integers
such that
\eq\sum_{i\in I}c_i a_{ij}= \sum_{i\in I} a_{ji} d_i =0
\quad\text{for all $j \in I$}.\label{eq:cd}
\eneq
Set
\eqn
\mathsf s_i \seteq c_i d_i^{-1} \in\Q_{>0} \ (i \in I), \quad c \seteq \sum_{i \in I} c_i h_i \in P^\vee,
\quad \delta \seteq \sum_{i \in I}  d_i \alpha_i \in P.
\eneqn
We call $c$ and $\delta$ by the \emph{canonical central element} and the \emph{null root}, respectively.
Note that we have
$$\lan c, \lambda \ran = (\delta, \lambda) \ \text{for any} \ \lambda \in P.$$
\noindent
Let us denote by $\g$ the affine Kac-Moody Lie algebra associated with
the affine Cartan datum $(A,P, \Pi,P^{\vee},\Pi^{\vee})$.
We denote by $\g_0$ the subalgebra of $\g$ generated by
$e_i, f_i, h_i$ for $i \in I_0$.
Then $\g_0$ is a finite-dimensional simple Lie algebra.

Let us denote by $U_q(\g)$ the quantum group associated with the affine Cartan datum $(A,P, \Pi,P^{\vee},\Pi^{\vee})$.
 We denote by $\uqpg$ the subalgebra of $U_q(\mathfrak{g})$ generated by $e_i,f_i,K_i^{\pm1}(i=0,1,,\ldots,n)$. We call $\uqpg$ the \emph{quantum affine algebra} associated with $(A,P, \Pi,P^{\vee},\Pi^{\vee})$.
Set
$$P_\cl=P/\Z\delta$$
and call it the {\em classical weight lattice}. Let $\cl\col P\to
P_\cl$ be the projection. Then $P_\cl=\soplus_{i\in I}\Z\cl(\La_i)$
and we have
$$P_\cl^\vee\seteq\Hom_{\Z} (P_\cl,\Z)=\set{h\in P^\vee}{\lan h,\delta\ran=0}
=\soplus_{i\in I}\Z h_i.$$ Set $\Pi_\cl = \cl(\Pi)$ and
$\Pi^{\vee}_\cl= \{h_0, \ldots, h_n\}$. Then $U_q'(\mathfrak{g})$
can be regarded as the quantum group associated with the quintuple
$(A,P_\cl, \Pi_\cl,P^{\vee}_\cl,\Pi^{\vee}_\cl)$.

Set $\varpi_i=\gcd(c_0,c_i)^{-1}(c_0\Lambda_i-c_i\Lambda_0) \in P$
for $i=1,2,\ldots,n$.
Then $\{\cl(\varpi_i)\}_{i \in I_0}$ forms a basis of
$P^0_\cl \seteq \set{\la\in P_\cl}{\lan c,\la\ran=0}$.
The Weyl group $W_0$ of $\g_0$ acts on $P^0_\cl$ in a natural way (\cite[\S 1.2]{AK}).

A $\uqpg$-module $V$ is called an {\em integrable module} if
 \bni
\item $V$ has a weight space decomposition
$$V = \bigoplus_{\lambda \in P_\cl} V_\lambda,$$
where  $V_{\lambda}= \set{ u \in V }{
\text{$K_i u =q_i^{\lan h_i , \lambda \ran} u$ for all $i\in I$}}$,
\item  the actions of
 $e_i$ and $f_i$ on $V$ are locally nilpotent for all $i\in I$.
\ee
We denote by $\CC_\g$ the category of finite-dimensional integrable $\uqpg$-modules.

For each $i \in I_0$, there exists a unique  integrable $\uqpg$-module $V(\varpi_i)$
satisfying the following conditions (\cite[\S\;1.3]{AK}):
\begin{enumerate}
\item  the weights of $V(\varpi_i)$ are contained in the convex hull of $W_0 \cl(\varpi_i)$,
\item $\dim V(\varpi_i)_{\cl(\varpi_i)} = 1$,
\item for any $\mu \in W_0 \cl(\varpi_i) \subset P^0_\cl$, we can associate a non-zero vector $u_\mu$ of weight $\mu$ such that
$$u_{ s_j \mu} = \begin{cases}
 f_j^{(\langle h_j, \mu \rangle )} u_\mu &
\text{if $\langle  h_j, \mu \rangle \ge 0$,} \\
 e_j^{(-\langle  h_j, \mu \rangle) } u_\mu &
\text{if $\langle h_j, \mu \rangle \leq 0$}
\end{cases}$$
for any $j\in I$.
\item $V(\varpi_i)$ is generated by $V(\varpi_i)_{\cl(\varpi_i)}$ as a  $\uqpg$-module.
\end{enumerate}
We call $V(\varpi_i)$ the {\em fundamental representation of $\uqpg$ of weight $\varpi_i$}.

Recall that the simple objects in $\CC_\g$ are parameterized by $I_0$-tuples of polynomials $\pi=(\pi_{i} (u) \; ; \; i \in I_0)$,
where $\pi_{i}(u) \in \cor[u]$ and $\pi_{i}(0)=1$ for $i \in I_0$ (\cite{CP94}, \cite{CP98}).
We denote by $\pi_{V,i} (u)$ the polynomials corresponding to a simple module $V$, and call $\pi_{V,i} (u)$  the {\it Drinfeld polynomials} of $V$.
 The Drinfeld polynomials $\pi_{V,i}(u)$ are determined by the eigenvalues
of the simultaneously commuting actions of
some {\it Drinfeld generators} of $\uqpg$ on a
subspace of $V$ (e.g., see \cite{CP94} for more details).

Assume that $\g$ is  an \emph{untwisted} affine Kac-Moody algebra.
Then the  Drinfeld generators depend on the choice of a
function $o \col I_0 \to \{-1,1\}$ such that $o(i)=-o(j)$ for adjacent
vertices $i,j$ in the  Dynkin diagram of $\g_0$.
If we fix such a function $o\col I_0 \to \{-1,1\}$, then by \cite[Proposition 3.1]{Nak042},  we have
\eq \label{eq:Drinfeld}
\pi_{V(\varpi_i),j}(u)= \begin{cases}
  1 & \text{if} \  j \neq i, \\
  1  + o(i)(-1)^{h}q^{-h^\vee} u & \text{if} \ j=i,
\end{cases}
\eneq
where $h\seteq \sum_{i \in  I} d_i$ and $h^\vee\seteq \sum_{i \in I} c_i$ denote  the Coxeter number and the dual Coxeter number, respectively.

A $\uqpg$-module $V \in \CC_\g$  is called \emph{good}, if $V$ admits a \emph{bar involution}, a \emph{crystal basis with simple crystal graph}, and a \emph{global basis}. For the precise definition of good modules see \cite{Kas02}.
For a good module $V$, there exists a non-zero weight vector $v \in V$ such that
$\wt(V) \subset \wt(v) + \sum_{i \in I_0} \Z_{\le 0} \cl (\alpha_i) $.
We call $\wt(v)$ the \emph{dominant extremal weight} and $v$  a \emph{dominant extremal weight vector}, respectively. A dominant extremal weight vector for a good module $V$ is unique up to constant multiple.
For example, the fundamental representations $V(\varpi_i)$ ($i \in I_0$) are good $\uqpg$-modules.

\section{Generalized quantum affine Schur-Weyl duality functors} \label{sec:SWfunctor}

In this section, we recall the functor constructed in \cite{KKK13}, which is a generalization of quantum affine Schur-Weyl duality functor.

\subsection{Generalized quantum affine Schur-Weyl duality functors}
Let $\uqpg$ be a quantum affine algebra over $\cor$ and let $\{V_s\}_{s\in \mathcal{S}}$ be a family of good $\uqpg$-modules.
For each $s \in \Ss$, let $\lambda_s$ be a dominant extremal weight of $V_s$ and let $v_s$ be
a dominant extremal weight vector in $V_s$ of weight $\lambda_s$.

Assume that we have an index set $J$ and two maps
 $X \colon J \rightarrow \cor^\times$,
$s \colon J \rightarrow \Ss$.

For each $i$ and $j$ in $J$, we have a $\uqpg$-module homomorphism
\eqn
\Rnorm_{V_{s(i)}, V_{s(j)}}(z_i,z_j) : (V_{s(i)})_\aff \tens (V_{s(j)})_\aff
\to \cor(z_i,z_j) \tens_{\cor[z^{\pm 1}_i,z^{\pm 1}_j]}(V_{s(j)})_\aff \tens (V_{s(i)})_\aff
\eneqn
which sends $v_{s(i)}\tens v_{s(j)}$ to $v_{s(j)} \tens v_{s(i)}$.
Here $(V_s)_\aff$ denotes the \emph{affinization} of the good module $V_s$, and
$z_i\seteq z_{V_{s(i)}}$ denotes the $\uqpg$-module automorphism on $(V_{s(i)})_\aff$ of weight $\delta$ (see, \cite[\S 2.2]{KKK13}).
We denote by $d_{V_{s(i)},V_{s(j)}}(z_j / z_i)$  the denominator of $\Rnorm_{V_{s(i)}, V_{s(j)}}(z_i,z_j)$,
which is the monic polynomial in $z_j/z_i$ of the smallest degree such that
\eqn
d_{V_{s(i)},V_{s(j)}}(z_j / z_i)\Rnorm_{V_{s(i)}, V_{s(j)}}(z_i,z_j) \bl (V_{s(i)})_\aff \tens (V_{s(j)})_\aff \br
\subset (V_{s(j)})_\aff \tens (V_{s(i)})_\aff.
\eneqn

We define a quiver $\Gamma^J$ associated with the datum $(J, X, s)$
 as follows:
\eq&&
\parbox{70ex}{\be[{(1)}]
\item we take $J$ as the set of vertices,

\item  we put $d_{ij}$ many arrows from $i$ to $j$,
where $d_{ij}$ denotes the order of the zero of
$d_{V_{s(i)},V_{s(j)}}(z_j / z_i)$ at $z_j / z_i = {X(j) / X(i)}$.
\ee}\label{gammaJ}
\eneq

We also define a symmetric Cartan matrix $A^J =(a^J_{ij})_{i,j\in J}$  by
\eq \label{eq:Cartan matrix}
a^J_{ij}=\begin{cases}
2&\text{if $i=j$,}\\
-d_{ij}-d_{ji}&\text{if $i\not=j$.}\end{cases}
\eneq
Note that $(a^J_{ij})_{i,j\in J}$ is the Cartan matrix obtained from the quiver $\Gamma^J$
following the way in \cite[\S 14.1.1]{Lus93} (see also \cite[\S 3.2.4]{R08}).

Set
\begin{equation}
P_{ij}(u,v) =(u-v)^{d_{ij}}.\label{def:Pij}
\end{equation}

Let $R^{J}( n )$  $(n \ge 0)$ be the symmetric \KLRs\  associated with the Cartan matrix $A^J$ and the parameters
\begin{equation}\label{eq:Q_omega}
Q_{ij}(u,v) = \delta(i\not=j)P_{ij}(u,v)P_{ji}(v,u)=\delta(i\not=j) (u-v)^{d_{ij}}(v-u)^{d_{ji}}
\quad (i,j \in J).
\end{equation}

We denote by $R^J(n) \gmod$ the category of graded $R^J(n)$-modules which is
finite-dimensional over $\cor$.

The following theorem is one of the main result of \cite{KKK13}.
\begin{theorem} \label{thm:gQASW duality}
 For each $n \in \Z_{\ge 0}$, there exists a functor
 $$\F_n  \colon R^J(n) \gmod \rightarrow \CC_\g \label{eq:the functor} $$
with the following properties:
\bna
\item
For each $i \in J$, let $S(\alpha_i)$ be the $1$-dimensional simple graded $R^J(1)$-module $\cor u(i)$ with the actions
\eq e(j) u(i)= \delta_{i,j} u(i), & x_1 u(i)=0. \label{eq:simple root module}\eneq
Then we have
\eqn \F_1(S(\alpha_i)) \simeq (V_{s(i)})_{X(i)}, \eneqn
 where $(V_{s(i)})_{X(i)}$ is the evaluation module of $V_{s(i)}$ at $z_i=X(i)$.
\item
Set
$$\F\seteq \soplus_{n \geq 0}\F_n \col\soplus_{n \geq 0} R^J(n) \gmod \rightarrow \CC_\g. $$
Then
$\F$
is a tensor functor.
Namely,
there exist canonical $\uqpg$-module isomorphisms
$\F(R^J(0))\simeq \cor$ and
\begin{equation*}
\F(M_1 \circ M_2) \simeq \F(M_1) \otimes \F(M_2)
\end{equation*}
for $M_1 \in R^J(n_1) \gmod$ and $M_2 \in R^J(n_2)\gmod$
and the diagrams in \cite[A.1.2]{KKK13} are commutative.
 \item
 If the Cartan matrix $(a_{i,j}^J)_{i,j \in J}$ is of type $A_n (n\ge 1), D_n (n \ge 4), E_6, E_7$, or $E_8$, then the functor $\F$ is exact.
 \ee
\end{theorem}
Note that $\F(qM) \simeq \F(M)$ for $M \in R^J(n) \gmod$.

\section{The category $\CC_Q$}

{}From now on, we assume that $\g$ is
a simply-laced affine Kac-Moody algebra;
 i.e., $\g$ is of type  $A^{(1)}_n \ (n \ge 1)$, $D^{(1)}_n \ (n \ge 4)$,
$E^{(1)}_6$, $E^{(1)}_7$, or $E^{(1)}_8$.
 We denote by $\g_0$ the  associated  finite-dimensional
semisimple Lie subalgebra of $\g$.
After reviewing the definition of the category $\CC_Q$
introduced in \cite{HL11},
we construct a \KLR\ $R^J$ associated with a set $J$ consisting of some simple modules in $\CC_Q$.
 Here $Q$ is a quiver whose underlying graph
is the Dynkin diagram of $\g_0$.
It turns out that
the underlying graph of the quiver $\Gamma^J$ coincides with the Dynkin diagram of $\g_0$ and
the functor $\F$ constructed  in \cite{KKK13} is  an exact functor
from $R^J\gmod$ to $\CC_Q$.
Moreover, the ring homomorphism induced by the functor $\F$ coincides with the homomorphism introduced in \cite[Theorem 1.2]{HL11}.

\subsection{Repetition quiver $\widehat Q$} \label{subsec:Qhat}
Let $I=\{ 0,1,\ldots, n \}$ and $I_0=\{ 1,\ldots, n\}$ be the index sets
of the simple roots of $\g$ and $\g_0$, respectively.
We denote by $W_0$ the Weyl group of $\g_0$ and by $w_0 \in W_0$ the longest element in $W_0$.
 We denote by $\Delta$,
$\Delta_+$ and $\Pi_0$ the set of roots, the one of
positive roots and the one of simple roots of $\g_0$, respectively.
 The standard invariant bilinear form of $\g_0$ on $\h_0^*$, the dual of the Cartan subalgebra of $\g_0$, is denoted by $( \ , \ )_0$; i.e., we have $(\al_i,\al_j)_0 = a_{i,j}$ for $i,j \in I_0$, where $(a_{i,j})_{i,j \in I_0}$ is the Cartan matrix of $\g_0$.

 Let $Q$ be a quiver whose underlying graph is the Dynkin diagram of $\g_0$.
 For $i \in I_0$, we denote by $s_i(Q)$ the quiver obtained from $Q$
by changing the orientation of every arrow with source $i$
or target $i$. We have $s_is_jQ=s_js_iQ$ for any $i,j\in I_0$.
A function $\xi \col I_0 \to \mathbb Z$ is called a {\it height function on $Q$} if
$\xi_j = \xi_i -1$ for $i \rightarrow j \in Q_1$, where $Q_1$ is the set of arrows of $Q$.
Since $Q$ is connected, any two height functions on $Q$ differ by a constant.
We fix such a function $\xi$.

Set
\begin{align*}
\widehat I_0 = \{(i,p) \in I_0 \times \mathbb Z \, ; \, p-\xi_i \in 2 \mathbb Z\}.
\end{align*}
The {\it repetition quiver $\widehat Q$ of $Q$} is the quiver with  $\widehat I_0$ as the set of vertices
 and its  arrows  are
\begin{align*}
(i,p) \rightarrow (j, p+1), \ (j,q) \rightarrow (i, q+1)
\end{align*}
for all $i \rightarrow j \in Q_1$ and for all $p, q$ such that $p-\xi_i, q-\xi_j \in 2 \mathbb Z$.
Note that $\widehat Q$ does not depend on the choice of orientation of $Q$.

A reduced expression $w=s_{i_1} s_{i_2} \cdots s_{i_\ell}$  of  an element $w$
in the Weyl group $W_0$ is called {\it adapted to $Q$}
if  $i_k$ is a source of $s_{i_{k-1}}  \cdots s_{i_2} s_{i_1} (Q)$
for $1 \le k  \le \ell $.
It is known that for each orientation $Q$, there exists a unique Coxeter element (a product of all simple reflections) $\tau \in W_0$ which is adapted  to  $Q$.

Set $\widehat \Delta \seteq \Delta_+ \times \mathbb Z$.
For $i \in I_0$, we define $\gamma_i = \sum_{j \in B(i)} \alpha_j,$
where $B(i)$ denotes the set of vertices $j$  in $Q$
such that there exists a path from $j$ to $i$.
We define a bijection $\phi \col \widehat I_0 \rightarrow \widehat \Delta$ inductively as follows.
  \eq &&
  \parbox{70ex}{
  \begin{enumerate}
    \item $\phi (i, \xi_i) = (\gamma_i, 0)$,
    \item if $\phi(i,p) = (\beta, m)$, then we define
    \be[{$\scbul$}]
      \item $\phi(i,p-2)=(\tau(\beta), m)$ if $\tau(\beta) \in \Delta_+$,
      \item $\phi(i,p-2)=(-\tau(\beta), m-1)$ if $\tau(\beta) \in \Delta_-$,
      \item $\phi(i,p+2)=(\tau^{-1}(\beta), m)$ if $\tau^{-1}(\beta) \in \Delta_+$,
      \item $\phi(i,p+2)=(-\tau^{-1}(\beta), m+1)$ if $\tau^{-1}(\beta) \in \Delta_-$.
\ee
  \end{enumerate}} \label{eq:phi}
\eneq

\subsection{Auslander-Reiten quiver $\Gamma_Q$}
 Let $\C Q \smod$ be the category of representations of the quiver $Q$ over $\C$ and let $\underline{\dim}(X)$
 denote the dimension vector of a representation $X$ of $Q$.
Then the full subquiver with vertices $\phi^{-1}(\Delta_+\times \{0\})$ in $\widehat Q$ is isomorphic to the {\em Auslander-Reiten quiver} $\Gamma_Q$ of $\C Q \smod$ and the vertex $\phi^{-1}(\beta,0)$ corresponds to the isomorphism class of  $M_Q(\beta)$, the indecomposable representation whose dimension vector is $\beta$ (\cite[\S 2.3]{HL11}).
In particular, $\gamma_i$ is the dimension vector of the injective envelope $I_Q(i)$ of the simple representation $S_Q(i)$ supported on vertex $i$.
 The dimension vector of the projective cover $P_Q(i)$ of $S_Q(i)$ is $\tau^{m_{i^{*}}}(\gamma_{i^{*}})$,
 where $^*$ is the involution of $I_0$ defined by $w_0{\alpha_i}=-\alpha_{i^{*}}$ and
 $m_{i} =\max\set{k \geq0}{\tau^k (\gamma_i) \in \Delta_+}$.

The arrows in $\Gamma_Q$ represent homomorphisms in a special class, so called \emph{irreducible morphisms}, and any non-isomorphism between indecomposable modules in $\C Q \smod$ is a sum of compositions of irreducible morphisms (\cite[Corollary IV 5.6]{ASS}).
It is known that $\Gamma_Q$ has no cycles and has no multiple arrows.
The subquiver of $\Gamma_Q$ with vertices $\{I_Q(i) \, ; \, i \in I_0\}$ is isomorphic to $Q^{{\rm rev}}$, the quiver obtained by reversing all the arrows of $Q$.
The isomorphism is given by $I_Q(i) \mapsto i$.
The subquiver of $\Gamma_Q$ with vertices $\{P_Q(i) \, ; \, i \in I_0\}$
is also isomorphic to $Q^{{\rm rev}}$ by mapping $ P_Q(i) \mapsto i$
(\cite[Proposition 6.4]{Gab}).
 If $\tau \beta \in  \Delta_+$ for $\beta \in \Delta_+$, then we set
$\tau M_Q(\beta) \seteq M_Q(\tau \beta)$. This map $\tau$ on $\Gamma_Q$ is called the {\it Auslander-Reiten translation}.

 The repetition quiver $\widehat Q$ itself is isomorphic to the Auslander-Reiten quiver
 of the category $D^b (\C Q \smod)$, the bounded derived category of $\C Q \smod$ (\cite{Ha}).
The isomorphism sends the vertex $\phi^{-1}(\beta, m)$
to the isomorphism class of the complex $M(\beta)[m]$, which is concentrated in degree $-m$.

\smallskip

We have the following description of the Auslander-Reiten quiver inside the repetition quiver $\widehat Q$:
\eq \label{eq:description GammaQ}
\phi^{-1}(\Delta_+ \times \{0\})
=\set{(i,p) \in \widehat I_0}{\xi_{i} -2m_{i} \le p \le \xi_i}.
\eneq
There exists a bijective map $\nu \col\widehat Q \to \widehat Q$,
so called {\it Nakayama permutation},
given by
\eqn \nu(i,p)=(i^*,p+h-2) \ \text{for} \  (i,p) \in \widehat I_0, \eneqn
 where $h=\sum_{i\in I}d_i$  denotes the Coxeter number of $\g_0$; i.e., $h=n+1$ for $A_n$,  $2n-2$ for $D_n$, $12$ for $E_6$, $18$ for $E_7$, and $30$ for $E_8$.
See \cite[\S\, 6.5]{Gab} and \cite[(3.4)]{KT}  with attention to the different labeling
of vertices of $\widehat Q$ from ours.
Since $\nu(\phi^{-1}(\underline \dim P(i^*),0)) = \phi^{-1}(\underline \dim I(i), 0)$ (\cite[Proposition 6.5]{Gab}),
 we obtain
\eq \label{eq:nakayama permutation}
\xi_{i^*}-2m_{i^*} = \xi_{i}-h+2 \quad\text{for $i \in I_0$.}
\eneq

\Lemma\label{lem:range}
We have
\eq&&
(i,\xi_j-d(i,j)),\ (i,\xi_j-2m_j+d(i,j))\in\phi^{-1}(\Delta_+ \times \{0\})
\label{eq:range}
\eneq
for any $i,j\in I_0$.
Here, $d(i,j)$ denotes the distance between $i$ and $j$ in the Dynkin diagram of $\g_0$,
\enlemma
\Proof
By \eqref{eq:description GammaQ} it is enough to show that
\eqn &&
\text{$\xi_{i}-2m_{i}=\xi_{i^*}-h+2  \le   \xi_{j} - d(i,j) \le \xi_i$, and}\\
&&\xi_{i}-2m_{i}=\xi_{i^*}-h+2  \le   \xi_{j}-2m_j+d(i,j)=\xi_{j^*}-h+2+d(i,j)
 \le \xi_i.
\eneqn
It follows from
$\xi_{i^*}-\xi_{j}+d(i,j)\le d(i^*,j)+d(i,j) \le h-2$ for any $i, j \in I_0$.
The last inequality $d(i^*,j)+d(i,j) \le h-2$
can be easily  verified in each type.
Indeed, if $\g_0$ is type $A_n$, then $d(i,j)+d(i^*,j)=d(i,i^*)$
or $d(j,j^*)$ which is  at most $n-1=h-2$.
If $\g_0$ is of the other type, then  $d(i,j)\le (h-2)/2$ for any $i,j\in I_0$.
\QED

For a reduced expression $s_{i_1}s_{i_2} \cdots s_{i_r}$ of $w_0$,
set $\beta_j \seteq s_{i_1}\cdots s_{i_{j-1}}(\alpha_{i_j})$ for $1 \le j \le r$.
Then we have $\{\beta_1, \ldots, \beta_{r} \} = \Delta_+$.
Moreover, if $1 \le k < \ell \le r$ and $\beta_k + \beta_\ell = \beta_j$ for some $1 \le j \le r$,
then we have $k < j < \ell$ (\cite{Papi94}).
We will use the following lemma later.
\Lemma \label{lem:p_k > p_l}
Assume that the reduced expression
$s_{i_1} s_{i_2} \cdots s_{i_r}$ of $w_0$  is adapted to $Q$.
If $k < \ell$ and $\beta_k+\beta_\ell \in \Delta_+$,  then there is a path from $\phi^{-1}(\beta_\ell,0)$ to $\phi^{-1}(\beta_k,0)$ in $\phi^{-1}(\Delta_+\times \{0\})$.
In particular, we have $p_k > p_\ell$.
Here $\phi(i_k,p_k)=(\beta_k,0)$ and $\phi(i_l,p_l)=(\beta_l,0)$.
\enlemma
\Proof
In this proof, we set $M(\beta) \seteq M_Q(\beta)$.
It is enough to show that there exists a path form $M(\beta_\ell)$ to $M(\beta_k)$ in $\Gamma_Q$.
By \cite[Proposition 4.12]{Lus90}, we have $\Hom_{\C Q}(M(\beta_i), M(\beta_j))=0$ whenever $1 \le i < j \le r$.

Note that $\Ext^1_{\C Q}(M(\beta_\ell), M(\beta_k))=0$. Indeed, if there is a non-split short exact sequence
$0\to M(\beta_k) \to M \to M(\beta_\ell) \to 0$ for some $M$, then we have
$\Hom_Q(M(\beta_k),M(\beta_p)) \neq 0$ and $\Hom_Q(M(\beta_p),M(\beta_\ell)) \neq 0$ for any indecomposable summand $M(\beta_p)$ of $M$.
It follows that $\ell \le p \le k$, which is a contradiction.

Recall that there is a bilinear form $\big< \ , \ \big>$ on the root lattice of $\g_0$ such that
$$\big( \underline{\dim}(X),\underline{\dim} (Y) \big)_0
= \big< \underline{\dim}(X),\underline{\dim} (Y) \big>
+ \big< \underline{\dim}(Y),\underline{\dim} (X) \big> \ \text{and}, $$
$$\big< \underline{\dim}(X),\underline{\dim} (Y) \big>
= \dim \Hom_{\C Q}(X, Y) -\dim \Ext^1_{\C Q}(X, Y)$$
for $X, Y \in \C Q \smod$.

Hence we have
\eqn
-1 = (\beta_\ell, \beta_k) &= \dim \Hom_{\C Q}(M(\beta_\ell), M(\beta_k)) -\dim \Ext^1_{\C Q}(M(\beta_\ell), M(\beta_k)) \\
& \quad +\dim \Hom_{\C Q}(M(\beta_k), M(\beta_\ell)) -\dim \Ext^1_{\C Q}(M(\beta_k), M(\beta_\ell)) \\
&= \dim \Hom_{\C Q}(M(\beta_\ell), M(\beta_k))-\dim \Ext^1_{\C Q}(M(\beta_k), M(\beta_\ell)),
\eneqn
so that
\eq \label{eq:extension} \dim \Ext^1_{\C Q}(M(\beta_k), M(\beta_\ell)) > 0.\eneq
In particular, $M(\beta_k)$ is not projective and hence $\tau(\beta_k) \in \Delta_+$.
By \cite[Corollary IV. 2.14]{ASS}, we have
\eqn
\dim \Hom_{\C Q}(M(\beta_\ell), M(\tau(\beta_k))) = \dim \Ext^1_{\C Q}(M(\beta_k), M(\beta_\ell))  > 0.
\eneqn

If $\Hom_{\C Q}(M(\beta_p),M(\beta_q)) \neq 0$ for $1 \le p \neq q \le r$,
 then we have a path from $M(\beta_p)$ to
$M(\beta_q)$ in  the Auslander-Reiten quiver $\Gamma_Q$,
because the arrows in $\Gamma_Q$ represent irreducible morphisms and any non-zero non-isomorphism between indecomposable modules in $\C Q \smod$ is a sum of compositions of irreducible morphisms.

Hence there exists a path from $M(\beta_\ell)$ to $M(\tau(\beta_k))$ in  $\Gamma_Q$.
Since there always exists a path (of length 2) from $M(\tau(\beta_k))$ to $M(\beta_k)$, we conclude that there is a path from $M(\beta_\ell)$ to $M(\beta_k)$ in $\Gamma_Q$.
\QED

Recall that the dimension vector is an {\it additive function} on $\Gamma_Q$  (\cite{Gab}); i.e.,
for each vertices $X \in \Gamma_Q$ such that $\tau(\underline \dim X) \in \Delta_+$, we have
\eq \label{eq:additive}
\underline{\dim} X +   \underline{\dim} (\tau X)
=\sum_{Z \in X^-} \underline{\dim} Z,
\eneq
where $X^-$ denotes the set of vertices $Z$ in $\Gamma_Q$ such that there exists an arrow from $Z$ to $X$.
For $Y \in \Gamma_Q$ let $Y^+$ denote the set of vertices $Z$ in $\Gamma_Q$ such that there exists an arrow from $Y$ to $Z$. Then we have $X^-=(\tau X)^+$ whenever $X$ is not projective (for example, see \cite[\S IV. 4]{ASS}).

The proof of following lemma is due to Bernard Leclerc.

\begin{lemma} \label{lem:boundary}
\bnum
\item
For each $k \in I_0$, the element $\phi^{-1}(\alpha_k,0)$ lies in the following subset of $\widehat I_0$:
  \begin{align}
    \partial \phi^{-1}(\Delta_+ \times \{0\}) \seteq&\{(i,\xi_i) \, ; \, i \in I_0\} \cup
    \{(i, \xi_i-2m_i)\, ; \, i \in I_0\} \\
\nonumber &\cup \{(i,\xi_i-2s) \, ; \, \text{$i$ is an extremal vertex of $I_0$}, \ 0 \le s \le m_i \},
  \end{align}
  where $m_i \seteq \max\{ s \geq 1 \, ; \, \tau^s (\gamma_i) \in \Delta_+ \}$.
\item
If $i\in I_0$ is not an extremal vertex and
$(i,p)\in\phi^{-1}(\Pi_0\times\{0\})$, then
$p=\xi_i$ or $p=\xi_{i}-2m_{i}$.
\ee
\end{lemma}
\begin{proof}
(i) If $k$ is a source of $Q$, then $\gamma_k = \alpha_k$ and
  $$\phi^{-1}(\alpha_k,0) = (k,\xi_k).$$

 If $k$ is a sink of $Q$, then $S_Q(k) \simeq P_Q(k) $.
  Hence we have $$\phi^{-1}(\alpha_k,0) = \phi^{-1}(\tau^{m_{k^*}}(\gamma_{k^*}),0)
   = (k^{*},\xi_{k^{*}}-2 m_{k^{*}}).$$
 Hence we may assume that the vertex $k$ is neither a source nor a sink.
We take a subquiver of type $A$ as follows:

\noi
Assume that $k$ has two neighbors. Then take an extremal vertex $i$ in $Q$ such that  between $k$ and $i$ there is no vertex with three neighbors.
Take the connected full subquiver of $Q$ whose extremal vertices are $i$ and $k$.
Let $k'$ be the only neighbor of $k$ in the subquiver.
Note that if $\g_0$ is of type $A$, then there are two choices for such subquivers.

Assume that $k$ has three neighbors. Then either $k$ has one incoming arrow and two outgoing arrows
or has two incoming arrows and one outgoing arrow.
Let $k'$ be a unique neighbor of $k$ such that there exists an arrow
$k'\to k$ (resp.\ $k\to k'$) in the first (resp.\ second) case.
Let $i$ be the extremal vertex of $Q$ such that $k'$ lies between $k$ and $i$.
Take the connected full subquiver of $Q$
whose extremal vertices are $i$ and $k$.

Thus we have chosen $i$ and $k'$ satisfying the following conditions:
\bnum
\item $i$ is an extremal vertex,
\item there is no vertex
with three neighbors between $k$ and $i$,
\item $k'$ is a neighbor of $k$ between $i$ and $k$,
\item one of the following two conditions holds:
\be[{\bf (a)}]
\item $k'$ is a unique vertex such that
there exists an arrow $k' \rightarrow k$.
\item
$k'$ is a unique vertex such that
there exists an arrow $k \rightarrow k'$.
\ee\ee

For the  case  {\bf (a)},
the following subquiver is contained in $\Gamma_Q$
by Lemma~\ref{lem:range}.
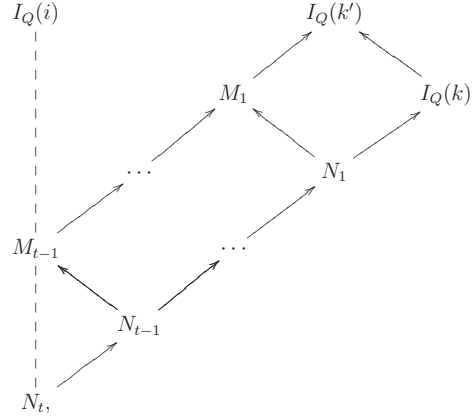
$$\scalebox{0.7}{\xymatrix{
I_Q(i)&&&I_Q(k')& \\
&&M_1 \ar[ur] &&I_Q(k) \ar[ul] \\
&\cdots \ar[ur]&  &N_1 \ar[ur] \ar[ul] & \\
M_{t-1} \ar[ur] \ar@{--}[uuu]\ar@{--}[dd]  &  &\cdots \ar[ur]&& \\
&N_{t-1} \ar[ur] \ar[ul] \ar[ur] \ar[ul]&&&\\
N_t, \ar[ur] &&&&\\
}}$$
for some indecomposable representations $M_1,\ldots, M_{t-1}, N_1 \ldots, N_t$ such that
$M_{t-1}$ and $N_{t}$ are in the $\tau$-orbit of $I_Q(i)$.
Here the vertex $\tau X$ is illustrated immediately below $X$.

Set $M_0 =I_Q(k')$ and $N_0=I_Q(k)$.
 Let us  calculate $\underline{\dim} N_t$ by \eqref{eq:additive}.
We have
$\tau(M_s)=N_{s+1} \ (0 \le s \le t-1)$, $(M_s)^-=\{M_{s+1}, N_s \} \ (0 \le s \le t-2)$
and $(M_{t-1})^-=\{N_{t-1}\}$.
By \eqref{eq:additive}, we obtain
\begin{align*}
\underline{\dim} N_t
&= \underline{\dim} N_{t-1} -\underline{\dim} M_{t-1} =
 \cdots = \underline{\dim} N_1 -  \underline{\dim} M_1 \\
 &= \underline{\dim} I_Q(k) -  \underline{\dim} I_Q(k')
 = \gamma_k - \gamma_{k'}=\alpha_k.
\end{align*}
The last equality follows from the fact that
$k'$ is a unique vertex such that there exists an arrow $k'\to k$.

Hence we have
$N_t=S_Q(k)$ and
$$\phi^{-1}(\alpha_k,0)=(i,\xi_k-t),$$
where $t$ denotes the number of edges between $k$ and $i$.

 In the  case  {\bf (b)},  we have
the following subquiver contained in $\Gamma_Q$
by Lemma~\ref{lem:range}:
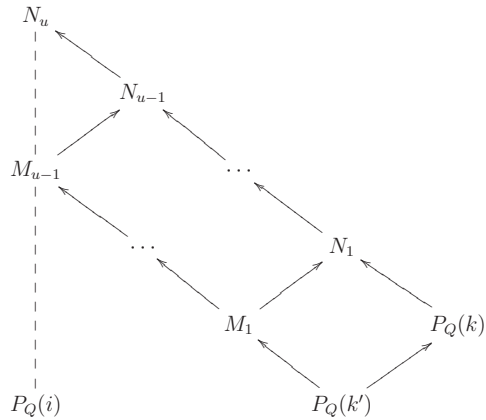
$$\scalebox{0.7}{\xymatrix{
N_u\ar@{--}[dd] &&&& \\
&N_{u-1} \ar[ul]&&& \\
M_{u-1} \ar[ur] \ar@{--}[ddd] &&\cdots \ar[ul]&& \\
&\cdots  \ar[ul]  & &N_1 \ar[ul]& \\
&& M_1 \ar[ur] \ar[ul] && P_Q(k)  \ar[ul] \\
P_Q(i)&&& P_Q(k') \ar[ul] \ar[ur] &\\
}}$$
for some indecomposable representations $M_1,\ldots, M_{u-1}, N_1 \ldots, N_u$ such that
$M_{u-1}$ and $N_{u}$ are in the $\tau$-orbit of $P_Q(i)$.
Set $M_0=P_Q(k')$ and $N_0=P_Q(k)$.
Then we have
$\tau(N_u)=M_{s-1} \ (1 \le s \le u)$, $(N_s)^-=\{N_{s-1}, M_s \} \ (1 \le s \le u-1)$
and $N_u^-=\{N_{u-1}\}$.
By \eqref{eq:additive}, we obtain
\begin{align*}
\underline{\dim} N_{u} &= \underline{\dim} N_{u-1} -\underline{\dim} M_{u-1} =
 \cdots = \underline{\dim} N_1 -  \underline{\dim} M_1 \\
 &= \underline{\dim} P_Q(k) -  \underline{\dim} P_Q(k') =\alpha_k.
 \end{align*}
 The last equation comes from the fact that
 ${\underline \dim} P_Q(k)=\sum_{j \in C(k)} \alpha_j$, where $C(k)$ is the set of vertices  $j$ in $Q$ such that there exists a path from $k$ to $j$.
Hence we have
$S_Q(k)= N_u$ and
$$\phi^{-1}(\alpha_k,0)=(i^{*},\xi_{k^{*}}-2 m_{k^*}+u), $$ where $u$ denotes the number of edges between $k$ and $i$.

\smallskip
\noi
 (ii) is immediate from the arguments above.
\end{proof}

We use the results in the appendix in the course of
the proof of the following lemma.
\begin{lemma} \label{lem:simple poles boundary}
  Let $\g$ be an affine Kac-Moody algebra of type $A^{(1)}_n \ (n \ge 1)$ or $D^{(1)}_n \ (n \ge 4)$.
    For  a Dynkin quiver $Q$ of type $\g_0$, let $\phi \col \widehat I_0 \rightarrow \widehat \Delta$ be the function defined in {\rm \S\,\ref{subsec:Qhat}}.
  If $(i,p)$ and $(j,r)$ are vertices in $ \phi^{-1}(\Pi_0 \times \{0\})$,
then the  normalized R-matrices
$\Rnorm_{V(\varpi_i),V(\varpi_j)}(z_i,z_j)$
has at most a simple pole at  $z_j /z_i = (-q)^{r-p}$.
\end{lemma}

\begin{proof}
If $\g$ is of type $A^{(1)}_n$, then all the normalized R-matrices between two fundamental representations have at most simple poles (for example, see \cite{DO94}).
 Hence we may assume that $\g$ is of type $D^{(1)}_n \ (n \ge 4)$.
We take the labeling
\eqref{Dynkin D}
of the vertices of the Dynkin diagram of type $D^{(1)}_n \ (n \ge 4)$.

Since $\Rnorm_{V(\varpi_i),V(\varpi_j)}(z_i,z_j)$ does not have a pole
at $z_j/z_i=(-q)^s$ for $s \le 0$,
we can assume $r > p$.
{}From the denominators in the appendix (Theorem~\ref{th:denomD}),
we know that
$\Rnorm_{V(\varpi_i),V(\varpi_j)}(z_i,z_j)$ has a double pole at $z_j/z_i= (-q)^s$
 if
\eqn
\quad2 \le i, j \le n-2,\quad i+j \ge n,\quad 2n-i-j \le s \le i+j,
\quad\text{and\quad$s \equiv i+j\;\mathrm{mod}\; 2$.}
\eneqn
Otherwise, $\Rnorm_{V(\varpi_i),V(\varpi_j)}(z_i,z_j)$ has at most a simple pole at $z_j/z_i= (-q)^s$.
Hence, without loss of generality, we may assume that
$2 \le i, j \le n-2$, $i+j \ge n$,
 and it is enough to show that $r-p < 2n-i-j$ or $r-p > i+j$.
In particular, $i$ and $j$ are not extremal vertices in $Q$.
On the other hand,
it is known that
$k^*=k$ and $m_k=n-2$ for all $k \in I_0$ (\cite[Section 6.5]{Gab}).
It follows that $(i,p) = \phi^{-1}(\alpha_i,0)$ and
$(j,r) = \phi^{-1}(\alpha_j,0)$ and $p \in \{\xi_i, \xi_i-(2n-4) \}$,
$r \in \{\xi_j, \xi_j-(2n-4) \}$ by Lemma~\ref{lem:boundary}~(ii).
{}From the definition of the height function $\xi$, we know that $|\xi_i-\xi_j| \leq |i-j|$ for any $1 \le i,j \le n-2$.

Now we have the following four cases:
  \bni
  \item  $p=\xi_i$ and $r=\xi_j$,
  \item  $p=\xi_i-(2n-4)$ and $r=\xi_j-(2n-4)$,
\item $p=\xi_i$ and $r=\xi_j-(2n-4)$,
  \item $p=\xi_i-(2n-4)$ and $r=\xi_j$,
  \ee

\noi
{\em Case} (i), (ii) : Assume that
$$2n-i-j \leq r-p=\xi_j-\xi_i. $$
Then we have $2n-i-j \leq |i-j|$
so that $\max(i,j) \geq n$, which  contradicts  $i,j \leq n-2$.
Hence we obtain $r-p < 2n-i-j$.

\smallskip\noi
{\em Case} (iii) :
we have $r - p = -2n+4+\xi_j - \xi_i \le -2n+4 +|i-j| \le -n$,
which  contradicts  $r-p > 0$.

\smallskip\noi
{\em Case} (iv) :
In this case,  $i$ is a sink of $Q$ and $j$ is a source of $Q$.
Assume that $$r-p=\xi_j-\xi_i+(2n-4) \leq i+j.$$
Then we have
$$2n-4   \leq i+j +\xi_i -\xi_j \le i+j +|i-j| = 2 \max(i,j) \le 2n-4$$
and hence $$2n-4  = i+j +\xi_i -\xi_j = 2 \max(i,j).$$
On the other hand, we have
$$2 \min(i,j) = i+j -|i-j| \le i+j+\xi_i-\xi_j=2n-4$$
so that
$i=j=n-2$. It contradicts the fact that $i$ is a sink of $Q$ and $j$ is a source of $Q$.
We conclude that $r-p > i+j$.
\end{proof}

\subsection{The category $\CC_Q$}

We fix  a function $o \col I_0 \to \{1, -1\}$ using the height function in \S \ref{subsec:Qhat} as
 \eqn o(i) \seteq -(-1)^{\xi_i} \ (i \in I_0). \eneqn
Then we have
\eqn
\pi_{V(\varpi_i),j}(u) =\begin{cases}
  1 & \text{if} \  j \neq i, \\
  1 -(-1)^{\xi_i}(-q)^{-h} u & \text{if} \ j=i,
\end{cases}
\eneqn
where $h$ denotes the Coxeter number of $\g$.

Set $\mathcal Y_\Z \seteq \Z[Y_{i,p}^{\pm 1} \ ; \ (i,p) \in \widehat I_0]$, the Laurent polynomial ring generated by  the indeterminates  $Y_{i,p}$'s.
An element $m$ in $\mathcal Y_\Z$ is called a {\it dominant monomial } if $m= \prod_{(i,p) \in \widehat I_0}  Y_{i, p}^{n_{i,p}}$ with $n_{i,p} \in \Z_{\ge 0}$ for all $(i,p) \in \widehat I_0$.

Let $m= \prod_{(i,p) \in \widehat I_0}  Y_{i, p}^{n_{i,p}}$ be a dominant monomial in $\mathcal Y_\Z$.
Following \cite{HL11}, we denote by $L(m)$ the simple module whose Drinfeld polynomials are given by \eqn
\pi_{L(m),i}(u)=
\prod_{\substack{p \in \Z \ \text{such that}  \\ (i,p) \in \widehat I_0}}
(1- q^{p} u)^{n_{i,p}} \qquad ( i \in I_0).
\eneqn
In particular, we have
\eqn
L(Y_{i,p}) \simeq V(\varpi_i)_{(-q)^{p+h}},
\eneqn
since $\pi_{(V_a), i}(u)= \pi_{V, i}(au)$ for  any simple  $V \in \CC_{\g}$, $i \in I_0$ and $a \in \cor^\times$ (see, for example \cite{CP95}).

Consider a partial order on $\widehat I_0$ given by
$(i,p) < (j,q)  \Leftrightarrow  p < q$.
Then for each dominant monomial $m= \prod_{(i,p) \in \widehat I_0}  Y_{i, p}^{n_{i,p}}$, the decreasingly ordered tensor product
\eqn
M(m) =\bigotimes_{(i,p) \in \widehat I_0}^{\leftarrow} L(Y_{i,p})^{n_{i,p}},
\eneqn
is well-defined up to isomorphism.
We call $M(m)$ the {\it standard module}  corresponding to the dominant monomial $m$.
By \cite[Theorem 9.2]{Kas02}, we know that $L(m)$ is isomorphic to the head of $M(m)$.

Let us define the category $\CC_{\mathbb Z}$ as the full subcategory of $\CC_\g$ whose objects $V$ satisfy that every composition factor $V$ is isomorphic to a module in
\eqn \Irr_\Z \seteq \set{L(m)}{ m \ \text{is a dominant monomial in } \mathcal Y_\Z}.
\eneqn
The category $\CC_\Z$ is closed under taking submodules, quotient modules, and extensions.

Recall that there exists an injective ring homomorphism
\eqn
\chi_q \colon \K(\CC_\g) \monoto \mathcal Y,
\eneqn
where $\K(\CC_\g)$ denotes the Grothendieck ring of $\CC_\g$ and
$\mathcal Y$ denotes the Laurent polynomial ring
$\Z[Y^{\pm1}_{i,a}]_{i \in I_0, a\in \cor^\times}$ which is generated by
the variables $Y^{\pm1}_{i,a}$ $(i \in I_0, a\in \cor^\times)$.
For $M \in \CC_\g$, we call $\chi_q(M)$ the {\it $q$-character} of $M$
(\cite{FR99, Her101}).
By identifying $Y_{i,p}$ with $Y_{i,q^p}$ for $(i,p) \in \widehat I_0$, we embed $\mathcal Y_\Z$ into $\mathcal Y$.

 Since  $\chi_q(M) \in \mathcal Y_\Z$ for every module  $M \in \CC_\Z$ (\cite[Proposition 5.8]{HL10}), we know that $\CC_\Z$ is closed under taking tensor products.
Hence $\Irr_\Z$ is the set of isomorphism classes of simple subquotients of tensor products of the modules in
\eqn
   \set{L(Y_{i,p})\simeq V(\varpi_i)_{(-q)^{p+h}}}{(i,p) \in \widehat I_0}.
\eneqn
It follows that the Grothendieck ring $\K(\CC_\Z)$ of $\CC_\Z$ is isomorphic to
the polynomial ring in
$\chi_q(L(Y_{i,p})) \ \bl (i,p) \in \widehat I_0 \br $.

Let $\mathcal Y_Q$ be the subring of $\mathcal Y_\Z$ generated by
$\set{Y_{i,p}^{\pm1}}{(i,p) \in  \phi^{-1}(\Delta_+ \times \{0\})}$.
Define the category $\CC_Q$ as the full subcategory of $\CC_{\mathbb Z}$ whose objects $V$ satisfy that
every composition factor of $V$ is isomorphic to a module in
\eqn \Irr_Q \seteq \set{L(m)}{ m \ \text{is a dominant monomial in } \mathcal Y_Q}. \eneqn
The category $\CC_Q$ is closed under taking submodules, quotient modules, and extensions.
By \cite[Lemma 5.8]{HL11}, $\CC_Q$ is closed under taking tensor products so that
$\Irr_Q$ is the set of isomorphism classes of simple subquotients of tensor products of the modules in
\eqn
   \set{L(Y_{i,p})\simeq V(\varpi_i)_{(-q)^{p+h}}}{(i,p) \in \phi^{-1}(\Delta_+ \times \{0\})}.
\eneqn

Let us denote by $\K(\CC_Q)$ the  Grothendieck ring of $\CC_Q$.
Then we have
\eqn
&&\K(\CC_Q) \isoto \Z[\chi_q(L(Y_{i,p})) \ ; \ (i,p) \in \phi^{-1}(\Delta_+ \times \{0\}) ].
\eneqn
{}From now on, we identify $\K(\CC_Q)$ with the above polynomial ring.

\section{The functor $\mathcal F$}
Let us keep the notations $\g$ and $\g_0$ as in the preceding sections.
In particular $\g$ and $\g_0$ are simply-laced.

\subsection{PBW basis and dual PBW basis}

We recall the $\Q(q)$-algebra automorphisms $T_i\col U_q(\g_0) \rightarrow U_q(\g_0)$ for $i \in I_0$ introduced in \cite{Lus90} which are given by
\eqn
&&T_i(q^h) = q^{s_i(h)}, \quad T_i(e_i) = -f_i K_i, \quad
T_i(f_i) = - K_i^{-1} e_i , \\
&&T_i(e_j) = \sum_{r+s=-a_{ij}} (-1)^r q^{-r} e_i^{(s)} e_j e_i^{(r)} \ \text{for} \ i \neq j, \\
&&T_i(f_j) = \sum_{r+s=-a_{ij}} (-1)^r q^{r} f_i^{(r)} f_j f_i^{(s)} \ \text{for} \ i \neq j.
\eneqn

Fix  a reduced expression
$w_0=s_{i_1} \cdots s_{i_r}$ of the longest element $w_0$ in the Weyl group $W_0$ of $\g_0$ and set $\beta_k\seteq s_{i_1} \cdots s_{i_{k-1}} (\alpha_{i_k}) \in \Delta_+$.
Then, for each $1 \le k \le r$,  the element
\eqn
F(\beta_k) \seteq T_{i_1} \cdots T_{i_{k-1}} (f_{i_k})
\eneqn
belongs to $U^-_\A(\g_0)$.
We call $F(\beta_k)$ the {\it root vector} corresponding to the root
$\beta_k$.
For each ${\bf c}=(c_1, \ldots ,c_r) \in \Z_{\ge 0}^r$, set
\eqn
&&F({\bf c}) \seteq F(\beta_1)^{(c_1)} \cdots F(\beta_r)^{(c_r)} \ \text{and} \
F^\upper({\bf c}) \seteq \dfrac{F({\bf c})}{(F({\bf c}),F({\bf c}))},
\eneqn
where  $x^{(c)} \seteq x^c / \prod_{k=1}^{c} \frac{q^k-q^{-k}}{q-q^{-1}}$ for
$x \in  U^-_\A(\g_0)$ and $c \in \Z_{\ge 0}$.
Since $(F({\bf c}),F({\bf c'}))=0$ for ${\bf c} \neq {\bf c'}$ (\cite[\S 38.2.3]{Lus93}), we have
$$(F^\upper({\bf c}),F({\bf c'})) = \delta_{{\bf c},{\bf c'}}.$$
 Note that
$$F^\upper({\bf c})=F^\upper(\beta_1)^{c_1} \cdots F^\upper(\beta_r)^{c_r},$$
where  $F^\upper(\beta_j)\seteq F(\beta_j) / (F(\beta_j),F(\beta_j) )$.
Let $\bf i$ be the sequence $(i_1, \ldots, i_r)$
corresponding to the reduced expression $s_{i_1} \cdots s_{i_r}$ of $w_0$.
As shown in \cite{Lus90} the sets
 \eqn
 &&P{}_{{\bf i}} \seteq \set{F({\bf c})}{ {\bf c} \in  \Z_{\ge 0}^r } \ \text{and} \
 P^\upper_{{\bf i}} \seteq \set{F^\upper({\bf c})}{ {\bf c} \in  \Z_{\ge 0}^r }
 \eneqn
are  $\A$-basis of $U^-_\A(\g_0)$ and $U^-_\A(\g_0)^\vee$, respectively.
  We call $P{}_{\bf i}$ and $P^\upper_{\bf i}$ the
{\it PBW basis of $U^-_q(\g_0)$} associated with $\bf i$ and
the {\it dual PBW basis of $U^-_q(\g_0)$} associated with $\bf i$,   respectively.
 The element  $F^\upper(\beta_j)$  is called the \emph{dual root vector} corresponding to the root $\beta_j$. We have $F^\upper(\beta_j) \in \B^\upper$.
It is not difficult to see that $\set{F^\upper(\beta_j)}{1 \le j \le r}$ generates $U^-_\A(\g_0)$ as an $\A$-algebra (for example, see \cite[\S 7.2]{GLS11b}).
The upper global basis can be characterized in terms of the dual PBW basis as follows.

\Prop[{\cite[Theorem 4.29]{Kimura12}}] \label{prop:transition mat}
For each $\bfc=(c_1,\ldots, c_r) \in \Z_{\geq 0}^r $, set
$$\KP({\bf c}) \seteq \set{{\bf a}=(a_1, \ldots, a_r) \in \Z_{\ge 0}^r}%
{\sum_{k=1}^r c_k \beta_k = \sum_{k=1}^r a_k \beta_k }.$$
We define a total order $>$ on $\Z^r_{\ge 0}$:
$\bfc'=(c'_1,\ldots,c'_r) > \bfc =(c_1,\ldots,c_r)$ if and only if
there exists $1 \le k \le r$ such that $c'_t = c_t$ for all $t < k$ and $c'_k > c_k$.

Then there exists a unique $\A$-basis
$$\set{B^\upper({\bf c})}{{\bf c} \in \Z_{\ge 0}^{r}}$$
of $U^-_\A(\g_0)$ with the following properties:
\eqn
&&\db(B^\upper({\bf c}))=B^\upper({\bf c}), \\
&&F^\upper({\bf c})= B^\upper({\bf c}) + \sum_{ {\bf a} \in \KP({\bf c}), \,   {\bf a} < {\bf c}} \kappa_{{\bf a},{\bf c}}(q) B^\upper(\bfa)
\eneqn
for some $\kappa_{\bfa, \bfc} (q) \in q\Z[q]$.
Moreover, we have
$$\set{B^\upper({\bf c})}{{\bf c} \in \Z_{\ge 0}^{r}} = \B^\upper.$$
\enprop

Note that if $\beta_j = \al_k$ for some $k \in I_0$, then
$F^\upper(\beta_j) = F^\upper(\bfc) = B^\upper(\bfc) = f_k $, where $\bfc = (c_1, \cdots, c_r) $
such that $c_u=0$  for $u \neq j$ and $c_j=1$.

\subsection{Hernandez-Leclerc homomorphism $\widetilde{\Phi}$}

The following theorem is a modification of Theorem 6.1 in \cite{HL11}.
\begin{theorem}[{\cite[Theorem 6.1]{HL11}}] \label{thm:HL}

Let $Q$ be a quiver whose underlying graph is the Dynkin diagram of $\g_0$,
and let $\phi\col\widehat I_0 \to \widehat \Delta$ be the bijection defined in \eqref{eq:phi}.
Fix a reduced expression  $w_0=s_{i_1} \cdots s_{i_r}$ of the longest element $w_0$ in the Weyl group $W_0$ of $\g_0$  adapted to $Q$.
 Then there exists a unique ring homomorphism $\widetilde{\Phi} \col U^-_\A(\g_0)^\vee \rightarrow \K(\CC_Q)$ given by
\eqn
   &\widetilde{\Phi}(q)= 1, \\
   &\widetilde{\Phi} (F^\upper(\beta_d))= [L(Y_{i,p})]\quad
\text{where $\phi(i,p) = (\beta_d,0)$.}
\eneqn
 The homomorphism $\widetilde{\Phi}$ is surjective and
$\Ker\widetilde{\Phi}= (q-1) U^-_\A(\g_0)^\vee$. Moreover,
$\widetilde \Phi$ sends the dual PBW basis and the upper global basis  to the set of isomorphism classes of standard modules and the one of simple modules, respectively.
\end{theorem}

\Proof
In \cite{HL11}, Hernandez and Leclerc constructed a $\C(t^{1/2})$-algebra isomorphism
$$\Phi \col \mathcal K_{t,Q} \isoto U_{\C(t^{1/2})}^+(\g_0)^\vee,$$
where $\mathcal K_{t,Q}$ denotes a $\C(t^{1/2})$-algebra whose specialization at $t^{1/2}=1$ is
isomorphic to  $\C \otimes_\Z \K(\CC_Q)$, and
$U_{\C(t^{1/2})}^+(\g_0)^\vee = \C(t^{1/2}) \otimes_\A U_\A^+(\g_0)^\vee$.
Here, the $\A$-form  $U_\A^+(\g_0)^\vee$ of $U_q^+(\g_0)$ is defined as in \cite[\S~6]{GLS11b} and  $\C[t^{\pm 1/2}]$ is an $\A$-algebra by the map $q \mapsto t$.

The algebra  $\mathcal K_{t,Q}$ has two distinguished $\C(t^{1/2})$-bases
\eqn
&&\set{\chi_{q,t}(L(m))}{m \ \text{is a dominant monomial in} \ \mathcal Y_Q } \ \text{and} \\
&&\set{\chi_{q,t}(M(m))}{m \ \text{is a dominant monomial in} \ \mathcal Y_Q }
\eneqn
(see, \cite[\S~ 5.6, 5.8 ]{HL11}).
Here, $\chi_{q,t}(V)$ denotes the {\it $(q,t)$-character} of $\uqpg$-modules $V$ in $\CC_\g$.
Moreover, the $\Z[t^{\pm 1}]$-submodules of $\mathcal K_{t,Q}$ generated by these bases coincide (\cite{Nak041}, see also \cite{HL11}). Let us denote it by $\mathcal K_{\Z[t^{\pm 1}],Q}$.
Then $\mathcal K_{\Z[t^{\pm 1}],Q}$ is a $\Z[t^{\pm 1}]$-subalgebra of $\mathcal K_{t,Q}$ (\cite{VV02}, see also \cite{HL11}).
On the other hand, the algebra $U_{\C(t^{1/2})}^+(\g_0)^\vee$ has two distinguished
$\C(t^{1/2})$-bases, so called the {\em rescaled upper global basis} and the {\em rescaled dual PBW basis}.
A rescaled upper global basis element is
an upper global basis element of $U^+_q(\g_0)$
up to a multiple of some powers of $t^{\pm 1/2}$. For their  precise
definitions, see \cite[\S ~6]{GLS11b}.  The rescaled dual PBW basis element is defined in a similar way
(see, \cite[\S~6.1]{HL11}).

Recall that we have a $\Q$-algebra isomorphism
\eq - \circ \vee \col U^-_\A(\g_0)^\vee \to
U^+_\A(\g_0)^\vee \eneq
given by $f_i \mapsto e_i$, $q \mapsto q^{-1}$.
By \cite[Lemma 6.1]{GLS11b}, we know that the $\Q$-algebra anti-isomorphism $- \circ \vee \circ *$
sends $\B^\upper$ to  the upper global basis  of $U_q^+(\g_0)$.
Since $\B^\upper$ is stable under the map $*$ (\cite{Kas91, Kas932}),
the $\Q$-algebra isomorphism $- \circ \vee$ also sends $\B^\upper$ to
 the upper global basis of $U_q^+(\g_0)$.
By \cite[Lemma 4.3.7]{KQ12}, we know that
$- \circ \vee $ sends  $P^\upper_{{\bf i}} $  to the set of dual PBW basis of $U^+_q(\g_0)$.
In particular, we have $\overline{(F^\upper(\beta_d))}^\vee =E^*(\beta_d) $ for all $1 \le d \le r$,
where $E^*(\beta_d)$ denotes the {\it dual PBW generator} of $U^+_q(\g_0)$ corresponding to the root $\beta_d$.
Under the isomorphism $\Phi$, the element $\chi_{q,t}(L(m))$ is mapped to a rescaled upper global basis element  and  $\chi_{q,t}(M(m))$ is mapped to a rescaled dual PBW basis element, respectively (\cite[Theorem 6.1]{HL11}).
In particular, we have
\eqn \Phi(\chi_{q,t}(L(Y_{i,p}))) =  t^{m} E^*(\beta_d), \eneqn
for some $m \in \dfrac{1}{2}\Z$  (see the proof of  \cite[Theorem 6.1]{HL11}).

Let ${\rm ev}_{t^{1/2}=1} \colon \mathcal K_{t,Q} \to \C \otimes_\Z \K(\CC_Q) $ be  the surjective $\C$-algebra homomorphism induced by mapping $t^{1/2} \mapsto 1$. Note that the $(q,t)$-characters of simple and standard modules in $\CC_Q$ are mapped to the $q$-characters of those under ${\rm ev}_{t^{1/2}=1}$.
Hence the upper global basis (respectively, dual PBW basis) elements in  $U_{\C(t^{1/2})}^+(\g_0)^\vee$ are mapped to the $q$-characters of simple (respectively, standard) modules under the homomorphism ${\rm ev}_{t^{1/2}=1} \circ \Phi^{-1}$. We  identify  the isomorphism class of a module in $\CC_\g$ with its $q$-character.
Then we obtain a surjective ring homomorphism
\eq {\rm ev}_{t^{1/2}=1} \circ \Phi^{-1}|_{U_\A^+(\g_0)^\vee} \col U_\A^+(\g_0)^\vee \to \K(\CC_Q)\eneq
whose kernel is $(q-1)U_\A^+(\g_0)^\vee $.
Finally, we define the map $\widetilde \Phi$ as
\eqn \widetilde \Phi= {\rm ev}_{t^{1/2}=1} \circ \Phi^{-1}|_{U_A^+(\g_0)^\vee} \circ - \circ \vee. \eneqn
Then $\widetilde \Phi$ satisfies all the desired properties.
\QED

\subsection{Homomorphism induced by the functor $\mathcal F$}\hfill

The following theorem is one of the main result in this section.

\begin{theorem} \label{thm:corresponding KLR}
Let $\g$ be a simply-laced affine Kac-Moody algebra, i.e.,
of type $A^{(1)}_n \ (n \ge 1)$, $D^{(1)}_n \ (n \ge 4)$, $E^{(1)}_6, E^{(1)}_7$ or $E^{(1)}_8$.
Fix  a reduced expression $w_0=s_{i_1} \cdots s_{i_r}$  of the longest element $w_0$ in the Weyl group $W_0$ of $\g_0$, which is  adapted to $Q$.

  Set
  \eq \label{eq:J}
  J \seteq \set{(i,p) \in \widehat I_0}{ \phi(i,p) \in \Pi_0 \times \{0\}},
  \eneq
  where $\phi$ is the map given in \eqref{eq:phi}.
  Let   $s \col J \to \set{V(\varpi_i)}{i \in I_0}$ and
  $X \col J \to \cor^\times $ be the maps given by
\eqn
  s(i,p) = V(\varpi_i), \ \text{and} \ X(i,p) = (-q)^{p+h} \ \text{for} \ (i,p) \in J.\eneqn
We assume the following conditions:
\eq
&&\parbox{75ex}{
for any $(i,p), (j,r) \in J$,
the normalized R-matrix $\Rnorm_{V(\varpi_i),V(\varpi_j)}(z)$ has at most
a simple pole at $z=(-q)^{r-p}$.}\label{conj:simple}
\eneq
  Then the corresponding Cartan matrix in \eqref{eq:Cartan matrix} to the datum $(J,s,X)$ is  of type $\g_0$.
  Moreover, we have a quiver isomorphism $ Q^{{\rm rev}} \isoto   \Gamma^J$  given by
  $$ s \mapsto \phi^{-1}(\alpha_s,0) \ \text{for} \ s \in I_0.$$
Here, $\Gamma^J$ is the quiver defined in \eqref{gammaJ}.
\end{theorem}

\begin{proof}
 Let $\alpha_s$ and $\alpha_t$ be two simple roots of $\g_0$.
Assume that
\begin{align}
 \beta_k= \alpha_s, \ \beta_\ell = \alpha_t, \ \text{and} \  k < l.
\end{align}
Set $(i_j, p_j) = \phi^{-1}(\beta_j, 0)$ for $1 \le j \le r$.
Recall that $V(\varpi_i)_{(-q)^{p+h}}
\cong L(Y_{i,p})$ for $(i,p) \in \widehat I_0$.
First, we will show that
\eq \label{asser:irreducible}
    &&\text{$L(Y_{{i_k,p_k}}) \otimes L(Y_{i_\ell,p_\ell})$ is irreducible
if and only if $(\alpha_s \, , \, \alpha_t)_0 = 0$.}
\eneq

Let $\bfc=(c_1, \ldots, c_r)$ be the element in $\Z_{\ge 0}^r$ given by
\eqn
c_u=0 \ \text{for} \ u \neq k,l, \quad c_k=c_l=1.
\eneqn
Then we have
\eqn
 F^\upper(\bfc) = F^\upper(\beta_k)F^\upper(\beta_\ell) =f_s f_t
\eneqn
It is not difficult to see that
\eqn
(\B^\upper)_{\alpha_t +\alpha_s} = \begin{cases}
 \Big\{ B^\upper(\bfc) =f_s f_t = f_tf_s \Big\}  & \text{if} \ (\al_s, \al_t)_0=0, \\
\Big\{ B^\upper(\bfc) = \dfrac{f_s f_t -q f_tf_s}{1-q^2},
\ B^\upper(\bfa) = \dfrac{f_t f_s -q f_s f_t}{1-q^2} \Big\}  & \text{if} \ (\al_s, \al_t)_0=-1,
\end{cases}
\eneqn
where $\bfa=(a_1, \ldots, a_r)$ is the only element in $\KP(\bfc) - \{\bfc\}$ in the case of $(\al_s, \al_t)_0=-1$.
It is explicitly given by
\eqn
a_u=0 \ \text{for} \ u \neq j, \quad a_j=1,
\eneqn
where $j $ is given by $\beta_j = \al_k+\al_l$.

Since $f_s f_t = B^\upper(\bfc) + q B^\upper(\bfa)$, applying $\widetilde \Phi$, we have
\eqn
[ L(Y_{{i_k,p_k}}) \otimes L(Y_{i_\ell,p_\ell}) ] =
\widetilde \Phi(f_s f_t) =\begin{cases}
 \widetilde \Phi(B^\upper(\bfc))  & \text{if} \ (\al_s, \al_t)=0, \\
\widetilde \Phi(B^\upper(\bfc)) +  \widetilde \Phi(B^\upper(\bfa)) & \text{if} \ (\al_s, \al_t)=-1.
\end{cases}
\eneqn
Since $\widetilde \Phi(B^\upper(\bfc))$ and  $\widetilde \Phi(B^\upper(\bfa))$ are isomorphism classes of non-isomorphic simple modules, we conclude that
$L(Y_{{i_k,p_k}}) \otimes L(Y_{i_\ell,p_\ell}) $ is irreducible if and only if $(\al_s, \al_t)_0=0$.
Hence the underlying graph of the quiver $\Gamma^J$ is the Dynkin diagram of $\g_0$.

For the orientation of $\Gamma^J$, observe that  $p_k > p_\ell$ by Lemma \ref{lem:p_k > p_l}.
Hence the arrow in the quiver $\Gamma^J$ between $(i_k,p_k)$ and $(i_\ell, p_\ell)$ is given by
$$\phi^{-1}(\al_s,0)= (i_k,p_k) \leftarrow (i_\ell, p_\ell)=\phi^{-1}(\al_t,0).$$

On the other hand, we have a non-split short exact sequence
$$0 \to M_Q(\beta_\ell) \to M_Q(\beta_\ell + \beta_k) \to M_Q(\beta_k) \to 0$$
in $\C Q \smod$ by \eqref{eq:extension}.
In particular, we have a non-zero homomorphism $M_Q(\al_s + \al_t) \to S_Q(s)$.
It implies
that the arrow between $s$ and $t$  in $Q$ is given by $s \to t$.
Hence we obtain the quiver isomorphism $ Q^{{\rm rev}} \isoto   \Gamma^J$  given by
$ s \mapsto \phi^{-1}(\alpha_s,0) \ (s \in I_0)$, as desired.

\end{proof}

We conjecture that
 \begin{conjecture} \label{conj:simple poles}
For any affine Kac-Moody algebra  of type $A^{(1)}_n \ (n \ge 1)$, $D^{(1)}_n \ (n \ge 4)$, $E^{(1)}_6, E^{(1)}_7$ or $E^{(1)}_8$, our assumption
\eqref{conj:simple} holds.
 \end{conjecture}
 By Lemma~\ref{lem:simple poles boundary},
 we already know that our conjecture holds for affine Kac-Moody algebras of type
  $A^{(1)}_n$ ($n \ge 1$) and $D^{(1)}_n$ ($n \geq 4$).

\medskip
Assume \eqref{conj:simple}.
Combining  Theorem \ref{thm:gQASW duality} and Theorem~\ref{thm:corresponding KLR},
we obtain an exact  functor
$$\F \col \soplus_{n \ge 0} R^J(n) \gmod \to \CC_\g$$
satisfying
\begin{align}
  \mathcal F (S(\al_t))= V(\varpi_i)_{(-q)^{p+h}} \simeq L(Y_{i,p}) \label{eq:haut 1},
\end{align}
where $\phi(i,p) = (\al_t, 0)$ ($t \in I_0$).
Since every module in $R^J \gmod$ can be obtained from $\{ S(\alpha_t) \, ; \, t \in I_0\}$ by taking convolution products, extensions, and subquotients, we have
\begin{align}
\mathcal F\col \soplus_{n \ge 0} R^J(n) \gmod\rightarrow \CC_Q.
\end{align}

A pair $(\beta_k, \beta_\ell)$ of positive roots is called a {\it minimal pair of $\beta_j$}
if $k < \ell$, $\beta_j=\beta_k +\beta_\ell$ and there exists no pair
$(\beta_{k'}, \beta_{l'})$ such that $\beta_j = \beta_{k'} + \beta_{l'}$ and $k < k' < j< l'< l$.
Note that if $|\beta_j| \ge 2$,
then there  exists  a minimal pair for $\beta_j$.
For a positive root $\beta_j$ of $\g_0$,  let $S(\beta_j)$ be
 the simple module in $R^J(\beta_j) \gmod$
such that $[S(\beta_j)]$ is mapped onto $F^\upper(\beta_j)$ under the isomorphism
$\K(R(\beta_j) \gmod) \isoto U^-_\A(\mathfrak g_0)^{\vee}_{-\beta_j}$.
We call the module $S(\beta_j)$ the \emph{dual root module of weight $\beta_j$} (see also \cite{Kato12}).

In \cite{McNa12, BKM12},  it is shown that the family
 $\{ S(\beta_j) \, ; \, j=1, \ldots ,r\}$ in $R^J(n) \gmod$ enjoys the following property:
    for each $\beta_j \in \Delta_+$ and a minimal pair $(\beta_k, \beta_\ell)$ of $\beta_j$, there exists an exact sequence in $R^J(\beta_j) \gmod$
\begin{align} \label{eq:minimal pair}
  \xymatrix{
   0 \ar[r]& q S(\beta_j) \ar[r] & S(\beta_k) \circ S(\beta_\ell) \ar[r]^-{f}&
   q^{-1}S(\beta_\ell) \circ S(\beta_k) \ar[r] & q^{-1}S(\beta_j) \ar[r] & 0.
    }
\end{align}
Moreover, we have
  \begin{align} \label{eq:simple image of rmat}
    \Img (f) \simeq \hd (S(\beta_k) \circ S(\beta_\ell)) \simeq q^{-1} \soc (S(\beta_\ell) \circ S(\beta_k))
  \end{align}
  and $\Img (f)$   is a simple $R^J(\beta_j)$-module.

\begin{remark}
\noindent
\bni
\item
The dual PBW basis of $U_q^+(\g_0)$ in \cite{BKM12} is the image of our dual PBW basis $P^\upper_{{\bf i}}$ of $U_q^-(\g_0)$ under the $\Q(q)$-algebra anti-isomorphism $\vee \circ * \col U_q^-(\g_0) \isoto U_q^+(\g_0)$.
More precisely, our $T_i$ is the same as $T^{''}_{i,1}$ in \cite[\S 37.1]{Lus93},  whereas
$T^{''}_{i,-1}$ is used to define the PBW basis in \cite{BKM12}. Since we have
$\vee \circ * \circ T^{''}_{i,1} \circ * \circ \vee = T^{''}_{i,-1}$, the above assertion follows.

\item The dual root module $L(\beta_j)$ in \cite{BKM12} is isomorphic to  $S(\beta_j)^\psi$, the module obtained from $S(\beta_j)$ by twisting the $R^J(|\beta_j|)$-action via the automorphism $\psi$ in \eqref{eq:auto psi}.

\item
In \cite{McNa12}, the first and the last term in \eqref{eq:minimal pair} are only described as some modules whose composition factors are isomorphic to $S(\beta_j)$ up to grading shift.
In \cite[Theorem 4.7]{BKM12}, it is proved that  they are simple.
\item
If $\beta_j=\al_i$ for some $i \in I_0$, then the module $S(\alpha_i)$ given by \eqref{eq:simple root module} is isomorphic to $S(\beta_j)$ so that the notation is  consistent.
\ee
\end{remark}

\begin{theorem} \label{thm:dual pbw generators}
For $j=1,\ldots, r$, we have
\begin{align*}
  \mathcal F(S(\beta_j)) \simeq  L(Y_{i_j,p_j}).
\end{align*}
 where $\phi(i_j,p_j)=(\beta_j,0)$.
\end{theorem}
\begin{proof}
  We will use induction on $ |\beta_j|$.
    When $ |\beta_j|=1$, the assertion follows from \eqref{eq:haut 1}.

  Assume that $|\beta_j| \geq 2$ and $(\beta_k, \beta_\ell)$ is a minimal pair for $\beta_j$ with $k < j <\ell$.
By induction hypothesis, we have $\mathcal F(S(\beta_k)) \simeq  L(Y_{i_k,p_k})$
and $\mathcal F(S(\beta_\ell)) \simeq L(Y_{i_\ell,p_\ell}) $.

By \eqref{eq:minimal pair} and \eqref{eq:simple image of rmat}, we have
\eqn
F^\upper(\beta_k) F^\upper(\beta_\ell) -q F^\upper(\beta_j) = [\hd(S(\beta_k) \circ S(\beta_\ell))]
 \eneqn
 under the isomorphism  $\K(R^J(\beta_j) \gmod) \isoto U^-_\A(\g_0)_{-\beta_j} $.
Applying $\tilde \Phi$, we obtain
\eq \label{eq:qchar of tensor}
 [L(Y_{i_k,p_k})  \otimes  L(Y_{i_\ell,p_\ell})] = [L]+ [L(Y_{i_j,p_j})]
\eneq
for some simple module $L$ in $\K(\CC_Q)$.
Since the module $L(Y_{i_k,p_k}) \otimes L(Y_{i_\ell,p_\ell})$
 should contain a simple subquotient isomorphic to $L(Y_{i_k,p_k}Y_{i_\ell,p_\ell})$,
 we conclude that
 $L \simeq L(Y_{i_k,p_k} Y_{i_\ell,p_\ell})$ and hence
  $L(Y_{i_k,p_k}) \otimes L(Y_{i_\ell,p_\ell})$ is not simple.

On the other hand,
applying $\F$ to \eqref{eq:minimal pair}, we obtain the following exact sequence:
\begin{align*}
\xymatrix@C=1pc{
   0 \ar[r]& \mathcal F(S(\beta_j)) \ar[r] & L(Y_{i_k,p_k}) \otimes L(Y_{i_\ell,p_\ell}) \ar[r]^-{\mathcal F(f)}&
 L(Y_{i_\ell,p_\ell}) \otimes L(Y_{i_k,p_k}) \ar[r] & \mathcal F(S(\beta_j)) \ar[r] & 0}.
    \end{align*}
Note that $\mathcal F(f)$ is non-zero. Indeed, if $\mathcal F(f)=0$, then we have
  $$ L(Y_{i_k,p_k}) \otimes L(Y_{i_\ell,p_\ell})  \simeq \F(S(\beta_j))
\simeq L(Y_{i_\ell,p_\ell}) \otimes L(Y_{i_k,p_k}) $$
so that
$L(Y_{i_k,p_k}) \otimes L(Y_{i_\ell,p_\ell})$ is irreducible, which is a contradiction.

By Lemma \ref{lem:p_k > p_l}, we  have $p_k > p_l$.
It follows that every non-zero homomorphism from $L(Y_{i_k,p_k}) \otimes L(Y_{i_\ell,p_\ell})$ to $L(Y_{i_\ell,p_\ell}) \otimes L(Y_{i_k,p_k})$ is a constant multiple of the normalized R-matrix.
Thus we have
$$\text{Im}(\mathcal F(f)) \simeq \hd(L(Y_{i_k,p_k}) \otimes L(Y_{i_\ell,p_\ell})) \simeq L(Y_{i_k,p_k}Y_{i_\ell,p_\ell}).$$
It follows that
$$[L(Y_{i_k,p_k}) \otimes L(Y_{i_\ell, p_\ell})]
     = [L(Y_{i_k,p_k} Y_{i_\ell,p_\ell})] + [\mathcal F (S(\beta_j))] \quad \text{in $\K(\CC_Q)$.}$$
Comparing with \eqref{eq:qchar of tensor}, we have
$[\mathcal F(S(\beta_j))] = [L(Y_{i_j,p_j})]$ and hence we conclude that
\begin{align*}
  \mathcal F(S(\beta_j)) \simeq L(Y_{i_j,p_j}),
\end{align*}
since $L(Y_{i_j,p_j})$ is simple.
\end{proof}

\begin{corollary}
  The ring homomorphism $ U^-_\A (\g_0)^{\vee} \rightarrow \mathcal \K(\CC_Q)$ induced  by the exact  functor
  \begin{align*}
\mathcal F \col \soplus_{n \ge 0}R^J(n)\gmod \To \CC_Q
  \end{align*}
  is the same as the homomorphism  $\widetilde{\Phi}$ in
{\rm Theorem \ref{thm:HL}}.
\end{corollary}
\begin{proof}
Theorem \ref{thm:dual pbw generators}  asserts
that two homomorphisms coincide on the set of dual root vectors.
 Since the dual root vectors generate $U^-_\A(\g_0)^\vee$, our assertion follows immediately.
  \end{proof}

Since $\widetilde \Phi$ sends the upper global basis to the set of isomorphism classes of simple modules in $\CC_Q$ (\cite{HL11}), and the upper global basis corresponds to the set of isomorphism classes of self-dual simple $R^J(n)$-modules (\cite{VV09}),  we obtain the following result.

\begin{corollary} \label{cor:simple to simple}
 The functor $\mathcal F_n \col R^J(n) \gmod \rightarrow \CC_Q$ sends the  simple modules to the simple modules.
\end{corollary}

\newcommand{\vpi}[2][]{V(\varpi_{#2})_{#1}}
\appendix
\section{The denominators of normalized $R$-matrices of type $D_n^{(1)}$ } \label{appen:simple poles}

\subsection{Affine  Kac-Moody algebra  of type $D_n^{(1)}$}
In this appendix, we will calculate the denominators of the normalized $R$-matrices between the fundamental representations of $U'_{q}(D_n^{(1)})$. We will use $\C(q)$ as the base field $\cor$.
Let $\g$ be the affine Kac-Moody algebra of type $D_n^{(1)}$ ($n \ge 4$).
The Dynkin diagram of $\g$ is given as follows:
\eq&&\ba{c}
\xymatrix@R=2ex@C=5ex{
*{\circ}<3pt> \ar@{-}[dr]^<{0}
&&&&&&&*{\circ}<3pt>
\ar@{-}[dl]^<{n-1}\\
&*{\circ}<3pt> \ar@{-}[r]_<{2} & *{\circ}<3pt> \ar@{-}[r]_<{3}
& {} \ar@{.}[r]&{} \ar@{-}[r]_>{\,\,\,\,n-3}
& *{\circ}<3pt> \ar@{-}[r]_>{n-2\,\,\,} &*{\circ}<3pt>\\
*{\circ}<3pt> \ar@{-}[ur]_<{1}&&&&&&&
*{\circ}<3pt> \ar@{-}[ul]_<{n}
}\ea
\label{Dynkin D}
\eneq
The weight lattice $P_\cl$ and  the dual weight lattice
$P_\cl^\vee$ for $\uqpg$ is given by \eqn P_\cl = \soplus_{i = 0}^n
\Z \Lambda_i, \quad P_\cl^\vee =  \soplus_{i = 0}^n \Z h_i. \eneqn
The null root and the center is given by
\eq &&\delta=\al_0+\al_1+2(\al_2+\cdots +\al_{n-2})+\al_{n-1}+\al_n,\\
&&c=h_0+h_1+2(h_2+\cdots +h_{n-2})+h_{n-1}+h_n,
\eneq
and $$\lan c,\rho\ran=2n-2.$$
Set $\Po=\set{\la\in P_\cl}{\lan c,\la\ran=0}$.
Then $\Po$ has a basis
$\vp_i$ ($1\le i\le n$)
where
$$\vp_i=\bc
\Lambda_i-\Lambda_0&\text{if $i=1,n-1,n$,}\\
\Lambda_i-2\Lambda_0&\text{if $2\le i\le n-2$.}
\ec$$
For an orthonormal basis $\{\epsilon_1, \ldots, \epsilon_n \}$
of $\Q \tens \Po$ with respect to the bilinear form induced from $(\cdot,
\cdot)$ on $P$,  we have
\begin{align*}
\alpha_i &=
\bc\epsilon_i -\epsilon_{i+1} & (1 \leq i \leq n-1), \\
   \epsilon_{n-1} + \epsilon_n & (i=n),\ec\\
\vp_i &=\bc
 \epsilon_1 + \cdots + \epsilon_i & (1 \leq i \leq n-2), \\
  \dfrac{1}{2}(\epsilon_1 + \cdots + \epsilon_{n-1} -\epsilon_n) & (i=n-1),\\[1ex]
  \dfrac{1}{2}(\epsilon_1 + \cdots + \epsilon_{n-1} + \epsilon_n) & (i=n).
  \ec
 \end{align*}
Let $\g_0$ be the
 Lie subalgebra of $\g$
whose Dynkin diagram is obtained by removing the $0$ vertex from
the Dynkin diagram of $\g$.
Then $\g_0$ is a simple Lie algebra of type $D_n$.
We may regard $\Po$ as a weight lattice of $\g_0$.
Let $W_0$ be the Weyl group of $\g_0$ and $w_0$ the longest element of $W_0$.

Let  $^*$ be the  automorphism of the Dynkin diagram
of $D_n$ arising from the action of the longest element $w_0 \in W_0$ given by
  $\varpi_{i^*}=-w_0 \varpi_i$ for $i \in I_0$.
  Then we have
$i^* = i$ for $ 1 \le i\le n-2$, and $(n-1)^* = n-1$, $n^*=n$ if $n$ is even,
$(n-1)^*=n$, $n^*=n-1$ if $n$ is odd.
For a finite-dimensional $\uqpg$-module $M$, we denote by $M^*$  the left dual of $M$
and by $^*M$ the
right dual of $M$; i.e., we have the following $\uqpg$-module homomorphisms:
\begin{align*}
M^* \otimes M \stackrel{{\rm tr}}{\longrightarrow} \cor,
\quad \cor \stackrel{\iota}{\longrightarrow} M \otimes M^* \quad \text{and} \quad
M \otimes {}^*M \stackrel{{\rm tr}}{\longrightarrow} \cor,
\quad \cor \stackrel{\iota}{\longrightarrow} {}^*M \otimes M.
\end{align*}
For example, we have
\begin{align*}
 \vpi{i}{}^* \simeq\vpi[q^{-2n+2}]{i^*}\quad \text{and} \quad
 ^*\vpi{i}{} \simeq\vpi[q^{2n-2}]{i^*}.
\end{align*}
We denote the denominator $d_{\vpi{i}{},\vpi{j}{}}(z)$
of $\Rnorm_{\vpi{i}{},\vpi{j}{}}(z)$ by $d_{i,j}(z)$.
By \cite[(A.6)]{AK} and
\cite[Proposition 6.15]{FM01} (see also \cite{Chari}),
we have
\eq
d_{i,j}(z)=d_{j,i}(z).\label{eq:symd}
\eneq

Our main result in this appendix is the following theorem.
\Th\label{th:denomD}
For $n\ge 4$, we have
\eqn
&&d_{k,l}(z) =\prod_{s=1}^{\min (k,l)}
(z-(-q)^{|k-l|+2s})(z-(-q)^{2n-k-l-2+2s}) \\
&&\hs{6.7ex}=\hs{-2ex}\prod_{\substack{|k-l|+2\le s\le k+l\\
s\equiv k+l\;\mathrm{mod}\,2}}\hs{-3ex}(1-(-q)^s)
\hs{-2ex}\prod_{\substack{2n-k-l\le t\le 2n-|k-l|-2\\
t\equiv k+l\;\mathrm{mod}\,2}}\hs{-3ex}(1-(-q)^t)\qquad\text{for $1\le k,l\le n-2$,}\\
&&d_{k,n-1}(z)=d_{k,n}(z) = \prod_{s=1}^{k}(z-(-q)^{n-k-1+2s})
\quad\text{for $1\le k\le n-2$,}\\
&&d_{k,l}(z)=\prod_{\substack{1\le s\le n-1,\; s\equiv k-l+1\,\mathrm{mod}\,2}}
(z-(-q)^{2s})\quad\text{for $n-1\le k,l\le n$.}
\eneqn
\enth

\subsection{Vector representation}
Let $V = \big(\soplus_{j=1}^n \C(q) v_j)
\oplus \big(\soplus_{j=1}^n \C(q) v_{\ol{j}})$
be a $(2n)$-dimensional vector space and let $B=\{1,2,\ldots, n, \bar n, \ldots, \bar 2, \bar 1\}$ be the
index set for the basis of $V$ with an ordering given by
$$1 \prec 2 \prec \cdots \prec n,\quad
\ol n \prec \cdots \prec \ol 2 \prec \ol 1,\quad n-1\prec \ol n,\ n\prec\ol{n-1}.$$
We define the $\uqpg$-module action on $V$ as follows:
\begin{align*} \allowbreak
  q^h v_j &= q^{\langle h, \wt(v_j) \rangle} v_j
  \quad \text{for $h \in (P_\cl)^\vee$ and $j \in B$,} \\ \allowbreak
  e_i v_j &=\begin{cases}
    v_i & \text{if $j= i+1$ and $i \neq 0,n$, }\\
    v_{\overline{i+1}} & \text{if $j= \overline{i}$ and $i \neq 0,n$, }\\
    v_n & \text{if $j= \ol{n-1}$ and $i = n$,}\\
       v_{n-1} & \text{if $j= \ol{n}$ and $i = n$,}\\
        v_{\ol{1}} & \text{if $j= 2$ and $i = 0$,}\\
         v_{\ol{2}} & \text{if $j= 1$ and $i = 0$,}\\
    0 & \text{otherwise},
  \end{cases} \\ \allowbreak
    f_i v_j &=\begin{cases}
    v_{i+1} & \text{if $j= i$ and $i \neq 0,n$,}\\
    v_{\overline{i}} & \text{if $j= \overline{i+1}$ and $i \neq0, n$,} \\
    v_{\overline{n}} & \text{if $j= n-1$ and $i = n$,}\\
    v_{\overline{n-1}} & \text{if $j= n$ and $i = n$,}\\
    v_{1} & \text{if $j= \ol{2}$ and $i = 0$,}\\
    v_{2} & \text{if $j= \ol{1}$ and $i = 0$,}\\
    0 & \text{otherwise},
  \end{cases}
 \end{align*}
where
\begin{align*}
  \wt(v_j) = \epsilon_j, \ \wt(v_{\overline j}) = -\epsilon_j & \quad
  \text{for $j =1,2, \ldots, n$.}
\end{align*}
Then $V$ is an integrable $\uqpg$-module,
called the {\em vector representation}. It is isomorphic to $V(\varpi_1)$
with a crystal graph
$$
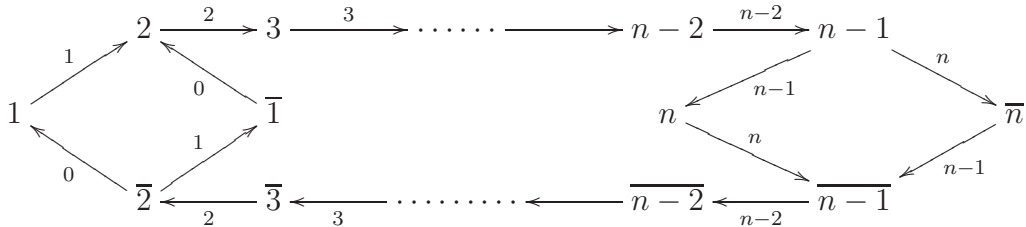
\xymatrix@C=7ex@R=3ex{
&2\ar[r]^-{2}&3\ar[r]^-{3}&\cdots\cdots\ar[r]
&n-2\ar[r]^-{n-2}&n-1\ar[dl]^{n-1}\ar[dr]^n\\
1\ar[ur]^-{1}&&\ol{1}\ar[ul]^-{0}&&n\ar[dr]^-{n}&&\ol{n}\ar[dl]^-{n-1}\\
&\ol{2}\ar[ul]^-{0}\ar[ur]^-{1}&
\ol{3}\ar[l]^-{2}&\cdots\cdots\cdots\ar[l]^-{3}&\ol{n-2}\ar[l]
&\ol{n-1}\ar[l]^-{n-2}
}
$$
Following \cite{DO96}, we recall the explicit form of the normalized $R$-matrix
\begin{align*}
  \Rnorm_{V(\varpi_1),V(\varpi_1)}(z) \col
   V(\varpi_1) \otimes V(\varpi_1)_z \rightarrow V(\varpi_1)_z \otimes V(\varpi_1)
\end{align*}
given by
\begin{align} \label{eq:rmatrix 11}
  \Rnorm_{V(\varpi_1),V(\varpi_1)}(z)(v_k \otimes v_\ell) =
  \begin{cases}
   v_k \otimes v_\ell & \text{if} \ k=\ell, \\[1.5ex]
   \dfrac{(1-q^2) z^{\delta(k \succ \ell)}}{z-q^2} (v_k \otimes v_\ell)
   + \dfrac{q (z-1)}{z-q^2} (v_\ell \otimes v_k)
    & \text{if} \ k \neq \ell, \overline \ell\\[3ex]
\dfrac{1}{(z-q^2)(z-q^{2n-2})} \sum_{j \in B} c_{jk}(z) (v_j \otimes v_{\overline j})
    & \text{if} \ k = \overline \ell,\\
  \end{cases}
\end{align}
where
\begin{align*}
  &c_{ij}(z) =
  \begin{cases}
    (q^2z-q^{2n-2})(z-1)
    & \text{if} \ i =j, \\
    (1-q^2)z(q^{|j|-|i|}(1-z) +\delta_{i,\overline j}(z-q^{2n-2})) & \text{if} \ i \prec j, \\
    (1-q^2)(q^{2n-2}q^{|j|-|i|}(1-z) +\delta_{i,\overline j}(z-q^{2n-2})) & \text{if} \ i \succ j, \\
  \end{cases} \\
& \text{and} \
|j|=\begin{cases}
  j & \text{if} \ j=1, 2, \ldots, n \\
  n-j & \text{if} \ j=\overline 1, \overline 2,\ldots, \overline n.
\end{cases}
\end{align*}
Here and in the sequel, we set
$$\ol{j}=k\quad\text{if $j=\ol{k}$ for $k=1,\dots,n$.}$$
Observe that the denominator $d_{1,1}(z) \seteq d_{\vpi{1}{},\vpi{1}{}}(z)$ of $\Rnorm_{V(\varpi_1),V(\varpi_1)}(z)$ is given by
\begin{align} \label{eq:deno11}
  d_{1,1}(z)=(z-q^2)(z-q^{2n-2}).
\end{align}
For each $k \geq 1$, set
\begin{align*}
 V^{\otimes (k)} &\seteq V_{(-q)^{1-k}} \otimes V_{(-q)^{3-k}} \otimes \cdots \otimes V_{(-q)^{k-3}} \otimes V_{(-q)^{k-1}}, \\
  \overline V^{\otimes (k)} &\seteq V_{(-q)^{k-1}} \otimes V_{(-q)^{k-3}} \otimes \cdots \otimes V_{(-q)^{3-k}} \otimes V_{(-q)^{1-k}},
\end{align*}
and define a $\uqpg$-module homomorphism
$$T^{(k)} \col V^{\otimes {(k)}} \rightarrow \overline V^{\otimes (k)}$$
by
$$T^{(k)} =\bl R_{1}(q^2) \circ \cdots \circ R_{k-1}(q^{2k-2})\br \circ \cdots
\circ \bl R_{1}(q^2) \circ R_{2}(q^{4})\br \circ \bl R_{1}(q^2)\br,$$
where
$R_i (x_{i+1}/x_i)$
denotes the homomorphism
\eqn
V_{x_1} \otimes \cdots \otimes V_{x_{i-1}} \otimes
d_{1,1}(z) \Rnorm_{V(\varpi_1),V(\varpi_1)}(z)|_{z=x_{i+1}/x_i}
\otimes
V_{x_{i+2}} \otimes \cdots \otimes V_{x_r}: \\
 V_{x_1} \otimes \cdots \otimes V_{x_i}
 \otimes V_{x_{i+1}} \otimes \cdots \otimes V_{x_r} \longrightarrow
 V_{x_1} \otimes \cdots \otimes V_{x_{i+1}}
 \otimes V_{x_i} \otimes \cdots \otimes V_{x_r}.
 \eneqn

\begin{prop}[{ \cite[Proposition 4.1-4.3]{DO96}}] \label{prop:fusion construction}
Let $2 \leq k \leq n-2$ and set
$W=\on{Im} \Rnorm_{V(\varpi_1),V(\varpi_1)}(q^{-2})\subset \ol{V}^{\otimes (2)}$.
  \bna
\item We have
\begin{align} \label{eq:kernel Tk}
  \Ker T^{(k)} = \sum_{j=0}^{k-2} \bl V^{\otimes(j)}\br_{(-q)^{j-k}}
  \otimes W_{(-q)^{2j+2-k}} \otimes
  \bl V^{\otimes(k-2-j)}\br_{(-q)^{j+2}}.
\end{align}
 \item The $\uqpg$-module $\Im (T^{(k)})$
 is isomorphic to
 the fundamental representation $V(\varpi_k)$ of weight $\varpi_k$.
  \ee
\end{prop}
Let us denote by $v_{[i_1,i_2,\ldots, i_k]} \in V(\varpi_k)$ the image of
 $v_{i_1} \otimes v_{i_2} \otimes \cdots \otimes v_{i_k}$ under the projection
 $V^{\otimes (k)} \twoheadrightarrow \vpi{k}{}$.

By \eqref{eq:kernel Tk}, we have
$$(\Ker T^{(k)})_{(-q)^{-\ell}} \otimes \bl V^{\otimes (\ell)}\br _{(-q)^k}
+\bl V^{\otimes (k)}\br_ {(-q)^{-\ell}}\tens (\Ker T^{(\ell)})_{(-q)^{k}}
\subset \Ker T^{(k+\ell)}.$$
Hence we obtain
\begin{corollary} \label{cor:k+l to k tensor l}
For $1 \leq k,\ell\leq n-2$ such that $k+\ell\le n-2$,
there exists a surjective $\uqpg$-module homomorphism
\begin{align}
  p_{k,\ell} \col V(\varpi_k)_{(-q)^{-\ell}} \otimes V(\varpi_\ell)_{(-q)^{k}}
  \epito V(\varpi_{k+\ell}).
  \end{align}
  Taking dual, we have an injective $\uqpg$-module homomorphism
\begin{align}
  \iota_{k,\ell} \col \vpi{k+\ell}{}  \mono \vpi[(-q)^{\ell}]{k}
  \otimes \vpi[(-q)^{-k}]{\ell}.
  \end{align}
\end{corollary}

\medskip

\subsection{Spin representations}

Let $V_{\rm sp}^{(+)}$ be the $\cor$-vector space with basis
\begin{align}
  B_{{\rm sp}}^{(+)} = \{ (m_1, \ldots, m_n) \, ; \, m_i = + \, \text{or} \, - , \ m_1 \cdots m_n = +\},
\end{align}
and
set $V_{\rm sp}^{(-)}$ be the $\cor$-vector space with basis
\begin{align}
  B_{{\rm sp}}^{(-)} = \{ (m_1, \ldots, m_n) \, ; \, m_i = + \, \text{or} \, - ,\ m_1 \cdots m_n =-\},
\end{align}
respectively. Define the $\uqpg$-actions on $V_{\rm sp}^{(\pm)}$ as follows:
for $v=(m_1,\ldots,m_n)$ and $i\in I$,
\begin{align*}
  q^{h} v &= q^{\langle h , \wt(v)\rangle} v
  \quad \text{for $h\in(P_\cl)^\vee$, where
  $\wt(v) = \frac{1}{2} \sum_{k=1}^n m_k \epsilon_k$,} \\
  e_iv &= \begin{cases}
    (m_1,\ldots, \overset{i}{+}, \overset{i+1}{-}, \ldots, m_n)
    &\text{if $i\not=0,n$ and $m_i=- $, $m_{i+1}=+$,} \\
    (m_1,\ldots, \overset{n-1}{+}, \overset{n}{+})
    &\text{if $i=n$ and $m_{n-1}=m_{n}=-$,} \\
    (-,-, m_3,\ldots, m_n)
    & \text{if $i=0$ and $m_1=m_2=+$,}\\
    0 & \text{otherwise},
           \end{cases} \\
  f_iv&= \begin{cases}
    (m_1,\ldots, \overset{i}{-}, \overset{i+1}{+}, \ldots, m_n)
    &\text{if $i\not=0,n$ and $m_i=+$, $m_{i+1}=-$,} \\
    (m_1,\ldots, \overset{n-1}{-}, \overset{n}{-})
    & \text{if $i=n$ and $m_{n-1}=m_{n}=+$,} \\
    (+,+, m_3,\ldots, m_n)
    & \text{if $i=0$ and $m_1=m_2=-$,}\\
    0 & \text{otherwise},
           \end{cases}
\end{align*}
We call $V_{\rm sp}^{(\pm)}$ by the {\it spin representations} of $\uqpg$.
They are simple $\uqpg$-modules and
\begin{align*}
  V(\varpi_{n-1}) \simeq V_{\rm sp}^{(-)},
  \quad V(\varpi_{n}) \simeq V_{\rm sp}^{(+)}
\end{align*}
as $\uqpg$-modules (\cite[Section 4.3.2]{KMN2}).

For $i,j \in B$ and $m \in B_{{\rm sp}}^{(\pm)}$, set
\begin{align*}
  m(i)=\begin{cases}
    (m_1,\ldots,\overset{i}{+},\ldots,m_n)
    & \text{if $i \preceq n$, $m_i=-$,} \\[2ex]
    (m_1,\ldots,\overset{\ol i}{-},\ldots,m_n) &
    \text{if $i \succ n$, $m_{\ol i}=+$,} \\[1ex]
    0 & \text{otherwise},
  \end{cases}
\end{align*}
and set $m(i,j)=\bl m(i)\br(j)$.
In \cite{DO96}, the following explicit form of the normalized $R$-matrices
are presented:

For $i \in B$, and $m \in B_{{\rm sp}}^{(+)}$,
\begin{align}
  \Rnorm_{1,n}(z) (v_i \otimes m) &=
  b_{i,m}(z) (m \otimes v_i)
  + \sum_{j \in B , j \neq i} b_{i,j,m}c_{i,j}(z)
  ((-q)^{n-1}z)^{\delta(i \succ j)} m(i,\bar j) \otimes v_j, \\
   \Rnorm_{n,1}(z) ( m \otimes v_i) &=b_{i,m}(z) (v_i\tens m)
 + \sum_{j \in B , j \neq i} \overline b_{i,j,m}c_{i,j}(z)
 ((-q)^{n-1}z)^{\delta(i \prec j)} v_j \otimes m(i,\bar j),\label{eq:rmatrix n1}
\end{align}
where
\begin{align*}
  b_{i,m}(z) &=
  \begin{cases}
    1 &\text{($i \preceq n$, $m_i=+$\quad or \quad$ i \succeq \overline n$,
    $m_{\bar i} =-$),} \\
    q \dfrac{z-(-q)^{n-2}}{z-{(-q)^{n}}}
    &\text{($i \preceq n$, $ m_i=-$\quad or \quad
    $i \succeq \overline n$, $ m_{\bar i}=+ $),}
  \end{cases} \\
  c_{i,j}(z) &=
      \dfrac{1-q^2}{z-(-q)^{n}},
\end{align*}
and for some non-zero constants $b_{i,j,m}$ and $\overline b_{i,j,m}$.
For the precise description of $b_{i,j,m}$ and $\overline b_{i,j,m}$, see \cite{DO96}.
Observe that
\begin{align} \label{eq:deno1n}
  d_{1,n}(z)=d_{n,1}(z)= z-(-q)^{n}.
\end{align}

The denominators of the normalized $R$-matrices between the spin representations are calculated in \cite{Okado90}:
\begin{align}
  d_{n-1,n-1}(z)=d_{n,n}(z)&=\prod_{s=1}^{[n/2]} (z-(-q)^{4s-2})
  =\prod_{\substack{2\le k\le 2n-2,\; k\equiv 2\,\mathrm{mod}4}} (z-(-q)^k),
  \label{eq:denonn}\\
   d_{n-1,n}(z)=d_{n,n-1}(z)&=\prod_{s=1}^{[(n-1)/2]} (z-(-q)^{4s})
  =\prod_{\substack{4\le k\le 2n-2,\; k\equiv0\,\mathrm{mod}4}} (z-(-q)^k)\label{eq:denon-1n},
\end{align}
where $[k]$ denotes the largest integer that does not exceed $k$.

The following proposition is proved in \cite{Koga98}.
\begin{prop} [{\cite[Proposition 3.1]{Koga98}}]
For $1 \leq k \leq n-2$ and $n',n''\in\{n-1,n\}$
such that  $n'-n''\equiv n-k\mod 2$, there exists a surjective homomorphism
\eq
  &&V(\varpi_{n'})_{(-q)^{-n+k+1}} \otimes
  V(\varpi_{n''})_{(-q)^{n-k-1}} \twoheadrightarrow V(\varpi_k),
  \label{eq:surjection1}\eneq
and an injective homomorphism
\eq
&&V(\varpi_k)\mono V(\varpi_{n'})_{(-q)^{n-k-1}}\tens
V(\varpi_{n''})_{(-q)^{-n+k+1}}. \label{eq:injection2}
\eneq
\end{prop}

The following lemma will be used in the next section.
\begin{lemma} \label{le:z-q^n-1}
Let $1 \leq k,l \leq n-2$  such that $k+ l = n-1$.
Then there exists a surjective $\uqpg$-module homomorphism
  \begin{align}
    V(\varpi_k)_{(-q)^{-l}} \otimes V(\varpi_l)_{(-q)^{k}}
    \epito V(\varpi_{n-1}) \otimes V(\varpi_{n}),\label{eq:nn-1epi}
  \end{align}
  and an injective $\uqpg$-module homomorphism
\begin{align}
   V(\varpi_{n-1}) \otimes V(\varpi_{n})
   \monoto
    V(\varpi_k)_{(-q)^{l}} \otimes V(\varpi_l)_{(-q)^{-k}}.\label{eq:nn-1mono}
  \end{align}
\end{lemma}

\begin{proof}
By \eqref{eq:injection2}, we obtain injective $\uqpg$-module homomorphisms
\eq
  &&\ba{l}
  \xi_1\col \vpi[(-q)^{-l}]{k}
  \monoto\vpi{n-1}\otimes \vpi[(-q)^{-n+k+1-l}]{n'}, \\[1ex]
   \xi_2\col\vpi[(-q)^{k}]{l}
   \monoto\vpi[(-q)^{n-l-1+k}]{n''}\otimes \vpi{n},
   \ea
 \eneq
  where $n'\in\{n-1,n\}$ such that $n-1-n'\equiv n-k \mod2$ and
  $n''=(n')^*$. Note that $n''-n\equiv n-l\mod2$.

Let $\vphi$ be the composition of the following $\uqpg$-module homomorphisms:
\eq&&\ba{c}
\xymatrix{
{\vpi[{(-q)^{-l}}]{k} \otimes \vpi[{(-q)^{k}}]{l}}\ar[d]^{\xi_1\tens \xi_2} \\
{\vpi{n-1}\otimes \vpi[(-q)^{-n+k+1-l}]{n'}\tens
\vpi[{(-q)^{n-l-1+k}}]{n''}\otimes \vpi{n}}
\ar[d]^{\vpi{n-1} \otimes {\rm tr}\otimes\vpi{n}}
\\
{\vpi{n-1}\otimes \vpi{n}.}}\ea \label{eq:pole k+1}
\eneq
We shall show that $\vphi$ does not vanish.
Let $v$ be a $U_q(\g_0)$-lowest weight vector of $\vpi[{(-q)^l}]{k}$
of weight $-\varpi_k$
and $w$ a $U_q(\g_0)$-highest weight vector of $\vpi[{(-q)^{k}}]{l}$
of weight $\varpi_\ell$.
Then by the crystal basis theory, we have, up to constant multiple
\eqn
\ba{lll}
 \xi_1(v)
&\equiv v_1\otimes u_{-\vp_{n''}}
&\mod \soplus_{\la\not=-\vp_{n''}}
\vpi{n-1}{}_{\vp_k-\la}\otimes \bl\vpi[(-q)^{-n+k+1-l}]{n'}\br_{\la}\\[1ex]
\xi_2(w)&\equiv u_{\vp_{n''}}\tens w_1
&\mod\soplus_{\la\not=\vp_{n''}}
\bl\vpi[{(-q)^{n-l-1+k}}]{n''}\br_{\la}\otimes \vpi{n},\ea
\eneqn
where $u_{-\vp_{n''}}$  is the lowest weight vector of
$\vpi[(-q)^{-n+k+1-l}]{n'}$ of weight $-\vp_{n''}$,\,
$v_1$ is a non-zero vector of $\vpi{n-1}$ of weight $-\vp_k+\vp_{n''}$,
$u_{\vp_{n''}}$ is the highest weight vector of
$\vpi[{(-q)^{n-l-1+k}}]{n''}$ and $w_1$ is a non-zero vector of $\vpi{n}$ of
weight $\vp_l-\vp_{n''}$.
Then we  have
$$\vphi(v\tens w)
\equiv\tr(u_{-\vp_{n''}}\tens u_{\vp_{n''}})v_1\tens w_1
\mod \soplus_{\la\not=\vp_k+\vp_{n''}}\vpi{n-1}{}_\la\tens
\vpi{n}{}_{\vp_k+\vp_l-\la}.$$
Hence $\vphi$ is non-zero.
Since $\vpi{n-1}\tens\vpi{n}$ is a simple $\uqpg$-module, $\vphi$ is surjective.
The homomorphism \eqref{eq:nn-1epi} is obtained from \eqref{eq:nn-1mono}
by duality.
\end{proof}

\subsection{Universal $R$-matrices.}

Let $K$ be a commutative ring containing $\cor[z, z^{-1}]$.
For a $\uqpg$-module $M$, we denote by $M_z$ the $K\tens \uqpg$-module
$K\tens_{\cor[z, z^{-1}]}M_\aff$. Here $z$ acts on $M_\aff$ by $z_M$.

A $\uqpg[z,z^{-1}]$-homomorphism from $M\otimes N_z$ to $N_z \otimes M$
is called  an {\it intertwiner}.

By taking $ \cor((z))$ for $K$, we have a $\uqpg$-module homomorphism
\begin{align*}
  \Runiv_{M,N}(z) : M \otimes N_z \rightarrow N_z \otimes M,
\end{align*}
functorially in finite-dimensional integrable $\uqpg$-modules $M$ and $N$.
It is called the {\it universal $R$-matrix from $M\otimes N_z$ to $N_z \otimes M$} (\cite{IFR92}).
Then the universal $R$-matrices satisfy the following property:
\begin{align*}
  \Runiv_{M,L\tens L'}(z) =(L_z \otimes \Runiv_{M,L'} (z))
  \circ (\Runiv_{M,L}(z) \otimes L'_{z}).
\end{align*}

Let $\lambda$ and $\mu$ be the dominant extremal weights of finite-dimensional irreducible $\uqpg$-modules $M$ and $N$, respectively.
Because the intertwiners between irreducible $\uqpg$-modules are unique up to constant (\cite[Corollary 2.5]{AK}), we have
\begin{align*}
  \Runiv_{M,N}(z) = a_{M,N}(z) \Rnorm_{M,N}(z),
\end{align*}
for some element $a_{MN}(z)$ of $ \cor((z))$.
Actually,  we have (see e.g.\ \cite[(A.4)]{AK})
\begin{align*}
a_{M,N}(z) \in q^{(\lambda | \mu)}(1+ z \cor[[z]]).
\end{align*}
The following useful lemma appeared in \cite{AK}.
\begin{lemma}[{\cite[Lemma C.15]{AK}}] \label{lem:dvw avw}
  Let $V', V'', V$ and $W$ be irreducible $\uqpg$-modules. Assume that
  there is a surjective morphism $V'\otimes V'' \epito V$. Then
  \begin{align*}
   \dfrac{ d_{W,V'}(z) d_{W,V''}(z) a_{W,V}(z)}{d_{W,V}(z)a_{W,V'}(z)a_{W,V''}(z)}
   \quad \text{and} &\quad
   \dfrac{ d_{V',W}(z) d_{V'',W}(z) a_{V,W}(z)}{d_{V,W}(z)a_{V',W}(z)a_{V'',W}(z)}
  \end{align*}
  are in $\cor[z,z^{-1}]$.
\end{lemma}

\subsection{The denominators of normalized R-matrices}
For $i,j \in I_0$, write
$$d_{i,j}(z)= \prod_{\nu}(z- x_\nu)
\ \text{and} \ d_{i^*,j}(z)= \prod_{\nu}(z- y_\nu).$$
Then by \cite[(a.10)]{AK}, \cite[(A.13)]{AK} and \eqref{eq:symd}, we have
\eq&&a_{i,j}(z)a_{i^* , j}(q^{-2n+2}z)\equiv
\dfrac{d_{i,j}(z)}{d_{i^* , j }(q^{2n-2}z^{-1})},\label{eq:ad}\\
 && a_{i,j}(z)
 = q^{(\varpi_i | \varpi_j)}
 \prod_{\nu}\dfrac{(p^* y_\nu z; p^{*2})_\infty
 (p^*\,  \overline{y_\nu}\, z; p^{*2})_\infty}
  {(x_\nu z; p^{*2})_\infty (p^{*2}\, \overline{x_\nu}\, z; p^{*2})_\infty},
  \nonumber
\eneq
where $p^*=q^{2n-2}$ and $(z;q)_\infty= \prod_{s=0}^\infty (1-q^nz)$.

Here and in the sequel, $f\equiv g$ means that
$f=ag$ for some $a\in\cor[z,z^{-1}]^\times$.
 Note that $\cor[z,z^{-1}]^\times = \set{c z^n}{n \in \Z, \ c \in \cor^\times}$.

Observe that $a_{i,j}(z)=a_{j,i}(z)$, because $d_{i,j}(z)=d_{j,i}(z)$
(see \eqref{eq:symd}).
By \eqref{eq:deno11} and \eqref{eq:deno1n}, we have
\begin{align}
 a_{1,1}(z) &\equiv \dfrac{[0][2n][2n-4][4n-4]}{[2][2n-2]^2[4n-6]}, \label{eq:a11}\\
 a_{1,n}(z) &\equiv  \dfrac{[n-2][3n-2]}{[n][3n-4]},\label{eq:an1}
\end{align}
where $[m]=((-q)^{m}z ; p^{*2})_\infty$.

Note that we have $[m]/[m+4n-4] \equiv z-(-q)^{-m}$.
\begin{prop}
  For $1\leq k \leq  n-2$, we have
  \begin{align} \label{eq:ank}
    a_{k,n}(z)\equiv \dfrac{[n-k-1][3n+k-3]}{[n+k-1][3n-k-3]}.
  \end{align}
\end{prop}
\begin{proof}
  Consider the following commutative diagram:
  \begin{align*}
    \xymatrix@R=7ex{
    \vpi{n} \otimes \vpi[(-q)^{-1}z]{k-1} \otimes \vpi[(-q)^{k-1}z]{1}
    \ar[r] \ar[d]_{\Runiv_{n,k-1}(-q^{-1}z) \otimes \vpi[(-q)^{k-1}z]{1}} &
    \vpi{n} \otimes \vpi[z]{k} \ar[dd]_{\Runiv_{n,k}(z)} \\
    \vpi[(-q)^{-1}z]{k-1} \otimes \vpi{n} \otimes \vpi[(-q)^{k-1}z]{1}
    \ar[d]_{\vpi[(-q)^{-1}z]{k-1} \otimes \Runiv_{n,1}((-q)^{k-1}z)} &\\
    \vpi[(-q)^{-1}z]{k-1}\otimes \vpi[(-q)^{k-1}z]{1}\otimes
    \vpi{n}\ar[r] &
    \vpi[z]{k} \otimes \vpi{n}.
      }
  \end{align*}
 Then we have
\begin{align*}
   \xymatrix@R=2.5ex{
    m^+_{n} \otimes v_{[1,\ldots, k-1]} \otimes v_{k}
    \ar@{|-_{>}}[r] \ar@{|-_{>}}[d] &
    m^+_{n} \otimes v_{[1,\ldots, k-1,k]} \ar@{|-_{>}}[dd]\\
    a_{n,k-1}(-q^{-1}z) v_{[1,\ldots, k-1]} \otimes m^+_{n} \otimes v_{k}
    \ar@{|-_{>}}[d] &\\
*++{a_{n,k-1}(-q^{-1}z) a_{n,1}((-q)^{k-1}z)
         v_{[1,\ldots, k-1]} \otimes w} \ar@{|-_{>}}[r] &
   a_{n,k}(z) v_{[1,\ldots, k-1,k]} \otimes m^+_{n},
    }
\end{align*}
where $m_n^+$ is the highest weight vector of $\vpi{n}$,
 and $v_{[1,\ldots,k]}$ is the highest weight vector of $\vpi{k}$,
 and $w=\Rnorm_{n,1}((-q)^{k-1}z)(m^+_{n} \otimes v_{k})$.
 Since $m^+_{n} \otimes v_{k}=f_{k-1}\cdots f_1(m^+_{n} \otimes v_{1})$
 and $v_{k} \otimes m^+_{n}=f_{k-1}\cdots f_1(v_{1} \otimes m^+_{n})$,
 we have
\begin{align*}
 w= \Rnorm_{n,1}((-q)^{k-1}z)(m^+_{n} \otimes v_{k})= v_{k} \otimes m^+_{n},
\end{align*}
and hence
\begin{align} \label{eq:recursive ank}
a_{n,k}(z) = a_{n,k-1}(-q^{-1}z) \ a_{n,1}((-q)^{k-1}z)
\end{align}
for $2 \leq k \leq n-2$.
By \eqref{eq:an1} and induction on $k$, we have the desired result.
\end{proof}

Now we can determine the denominators between the spin representations and the other fundamental representations.

\begin{theorem}
  For $1 \leq k \leq n-2$, we have
  \begin{align}
    d_{k,n}(z) = \prod_{s=1}^{k}(z-(-q)^{n-k-1+2s}).
  \end{align}
\end{theorem}
\begin{proof}
The case $k=1$ is already proved in \eqref{eq:deno1n}. Assume that $k \geq 2$.
Applying Lemma \ref{lem:dvw avw} to the surjection
$\vpi[(-q)^{-1}]{k-1} \otimes \vpi[(-q)^{k-1}]{1}\rightarrow \vpi{k}$, we have
  \begin{align*}
    \dfrac{d_{k-1,n}(-q^{-1}z) d_{1,n}((-q)^{k-1}z) a_{k,n}(z)}
  { d_{k,n}(z)a_{k-1,n}(-q^{-1}z) a_{1,n}((-q)^{k-1}z)} \in \cor[z,z^{-1}].
  \end{align*}
  By \eqref{eq:recursive ank}, we obtain
\begin{align} \label{eq:dkn divides dk-1n dn1}
    \dfrac{d_{k-1,n}(-q^{-1}z) d_{n,1}((-q)^{k-1}z) }
  {d_{k,n}(z)}
 \equiv\dfrac{d_{k-1,n}(-q^{-1}z) (z-(-q)^{n-k+1}) }
  {d_{k,n}(z)}
    \in \cor[z,z^{-1}].
  \end{align}

On the other hand, for each $2 \le k \le n-2 $, we have a surjective homomorphism
\begin{align*}
V(\varpi_k)_{-q^{-1}} \otimes V(\varpi_1)_{(-q)^{2n-2-k}}
\epito V(\varpi_{k-1}),
  \end{align*}
 given by the composition
\begin{align*}
V(\varpi_k)_{-q^{-1}} &\otimes V(\varpi_1)_{(-q)^{2n-2-k}}
\To[{\iota_{1,k-1}}] \\
&\vpi{k-1}{} \otimes \vpi[(-q)^{-k}]{1}  \otimes V(\varpi_1)_{(-q)^{2n-2-k}}
\To[{\rm tr}] V(\varpi_{k-1}).
  \end{align*}
Applying Lemma \ref{lem:dvw avw} to this surjection,
we obtain
    \begin{align*}
    \dfrac{d_{1,n}((-q)^{k+2-2n}z) d_{k,n}(-qz) a_{k-1,n}(z)}
  { d_{k-1,n}(z)a_{1,n}((-q)^{k+2-2n}z) a_{k,n}(-qz)} \in \cor[z,z^{-1}].
  \end{align*}
By \eqref{eq:ank}, we have
\begin{align*}
    \dfrac{ a_{k-1,n}(z)}
  {a_{1,n}((-q)^{k+2-2n}z) a_{k,n}(-qz) }
  \equiv
  \dfrac{z-(-q)^{n-k-2}}{z-(-q)^{n-k}} &\quad \mod \ \cor[z,z^{-1}]^\times.
\end{align*}
Thus we obtain
    \begin{align}
    \dfrac{ d_{k,n}(-qz) }
  {d_{k-1,n}(z)}
  \dfrac{(z-(-q)^{3n-k-2})(z-(-q)^{n-k-2})}{(z-(-q)^{n-k})} \in \cor[z,z^{-1}].
  \end{align}
By the induction hypothesis, we know that
$d_{k-1,n}(z)$ does not vanish at $z=(-q)^{3n-k-2}$, $(-q)^{n-k-2}$.
It follows that
  \begin{align*}
    \dfrac{ d_{k,n}(-qz) }
  {d_{k-1,n}(z)(z-(-q)^{n-k})} \in \cor[z,z^{-1}],
  \end{align*}
which is equivalent to
\begin{align*}
    \dfrac{ d_{k,n}(z) }
  {d_{k-1,n}(-q^{-1}z)(z-(-q)^{n-k+1})}\in \cor[z,z^{-1}].
  \end{align*}

Comparing with  \eqref{eq:dkn divides dk-1n dn1}, we have
\begin{align*}
  d_{k,n}(z)\equiv d_{k-1,n}(-q^{-1}z)(z-(-q)^{n-k+1})
  =\prod_{s=1}^{k}(z-(-q)^{n-k-1+2s}),
\end{align*}
as desired.
\end{proof}

To determine $d_{k,l}(z)$ for $1 \leq k,l \leq n-2$, we  first calculate
$a_{k,l}(z)$.
\begin{prop}
  For $1 \leq k,l \leq n-2$, we have
  \begin{align}
    a_{k,l} (z) \equiv \dfrac{[|k-l|][2n+k+l-2][2n-k-l-2][4n-|k-l|-4]}
    {[k+l][2n+k-l-2][2n-k+l-2][4n-k-l-4]}.
  \end{align}
\end{prop}

\begin{proof}
The case $k=l=1$ is already proved in \eqref{eq:a11}.

 We shall first calculate $a_{1,k}(z)$.
 Assume that $2 \leq k \leq n-2$.
 Consider the following commutative diagram:
   \begin{align*}
    \xymatrix@R=7ex{
    \vpi{1}{} \otimes \vpi[(-q)^{-1}z]{k-1} \otimes \vpi[(-q)^{k-1}z]{1}
    \ar[r] \ar[d]_{\Runiv_{1,k-1}(-q^{-1}z) \otimes \vpi[(-q)^{k-1}z]{1}} &
    \vpi{1} \otimes \vpi[z]{k} \ar[dd]_{\Runiv_{1,k}(z)} \\
    \vpi[(-q)^{-1}z] {k-1}\otimes \vpi{1} \otimes \vpi[(-q)^{k-1}z]{1}
    \ar[d]_{\vpi[(-q)^{-1}z]{k-1} \otimes \Runiv_{1,1}((-q)^{k-1}z)} &\\
    \vpi[(-q)^{-1}z]{k-1} \otimes \vpi[(-q)^{k-1}z]{1}\otimes \vpi{1} \ar[r] &
    \vpi[z]{k} \otimes \vpi{1}
    }
  \end{align*}
 Then we have
\begin{align*}
   \xymatrix@R=3ex{
    v_1 \otimes v_{[1,\ldots, k-1]} \otimes v_{k}
    \ar@{|-_{>}}[r] \ar@{|-_{>}}[d] &
    v_1 \otimes v_{[1,\ldots, k-1,k]} \ar@{|-_{>}}[dd]\\
    a_{1,k-1}(-q^{-1}z) v_{[1,\ldots, k-1]} \otimes v_1 \otimes v_{k}
    \ar@{|-_{>}}[d] &\\
*++{a_{1,k-1}(-q^{-1}z) a_{1,1}((-q)^{k-1}z)
         v_{[1,\ldots, k-1]} \otimes w} \ar@{|-_{>}}[r] &
   a_{1,k}(z) v_{[1,\ldots, k-1,k]} \otimes v_1,
    }
\end{align*}
where $w=\Rnorm_{1,1}((-q)^{k-1}z)( v_1 \otimes v_{k})$.
 By \eqref{eq:rmatrix 11}, we have
\begin{align*}
 w= \Rnorm_{1,1}((-q)^{k-1}z)(v_1 \otimes v_{k})=
 \dfrac{q((-q)^{k-1}z -1 )}{(-q)^{k-1}z -q^2}v_{k} \otimes v_1
+\dfrac{1-q^2}{(-q)^{k-1}z -q^2}v_{1} \otimes v_k,
\end{align*}
and hence
\begin{align*}
a_{1,k}(z) &\equiv a_{1,k-1}(-q^{-1}z) \ a_{1,1}((-q)^{k-1}z)
 \dfrac{q((-q)^{k-1}z -1 )}{(-q)^{k-1}z -q^2} \\
&\equiv a_{1,k-1}(-q^{-1}z) \ a_{1,1}((-q)^{k-1}z)
\dfrac{[k-1][4n+k-7]}{[k-3][4n+k-5]}.
\end{align*}
for $2 \leq k \leq n-2$. Then, by induction on $k$, we have
\begin{align} \label{eq:a1k}
  a_{1,k}(z)= a_{k,1}(z) \equiv\dfrac{[k-1][2n+k-1][2n-k-3][4n-k-3]}{[k+1][2n+k-3][2n-k-1][4n-k-5]}.
\end{align}

\medskip
Now we shall calculate $a_{k,l}(z)$ for
$2 \leq l \leq k \leq n-2$.
Consider the following diagram:
   \begin{align*}
    \xymatrix@R=6.5ex{
    \vpi{k} \otimes \vpi[(-q)^{-1}z]{l-1} \otimes \vpi[(-q)^{l-1}z]{1}
    \ar[r] \ar[d]_{\Runiv_{k,l-1}(-q^{-1}z) \otimes \vpi[(-q)^{l-1}z]{1}} &
    \vpi{k} \otimes \vpi[z]{l}\ar[dd]_{\Runiv_{k,l}(z)} \\
    \vpi[(-q)^{-1}z] {l-1}\otimes \vpi{k} \otimes \vpi[(-q)^{l-1}z]{1}
    \ar[d]_{\vpi[(-q)^{-1}z]{l-1}\otimes \Runiv_{k,1}((-q)^{l-1}z)} &\\
    \vpi[(-q)^{-1}z]{l-1} \otimes \vpi[(-q)^{l-1}z]{1} \otimes \vpi{k}\ar[r] &
    \vpi[z]{l} \otimes \vpi{k}
    }
  \end{align*}
 Then we have
\begin{align*}
   \xymatrix@R=3.5ex{
    v_{[1,\ldots,k]} \otimes v_{[1,\ldots, l-1]} \otimes v_{l}
    \ar@{|-_{>}}[r] \ar@{|-_{>}}[d] &
    v_{[1,\ldots,k]} \otimes v_{[1,\ldots, l-1,l]} \ar@{|-_{>}}[dd]\\
    a_{k,l-1}(-q^{-1}z) v_{[1,\ldots, l-1]} \otimes v_{[1,\ldots,k]} \otimes v_{l}
    \ar@{|-_{>}}[d] &\\
*++{a_{k,l-1}(-q^{-1}z) a_{k,1}((-q)^{l-1}z)
         v_{[1,\ldots, l-1]} \otimes w} \ar@{|-_{>}}[r] &
   a_{k,l}(z) v_{[1,\ldots, l-1,l]} \otimes v_{[1,\ldots,k]},
    }
\end{align*}
where $w=\Rnorm_{k,1}((-q)^{l-1}z)( v_l \otimes v_{[1,\dots,k]})$.
 Since
\begin{align*}
f_{l-1} f_{l-2} \cdots f_1 (v_{[1,\dots,k]} \otimes v_1) =   v_{[1,\dots,k]} \otimes v_l,\\
f_{l-1} f_{l-2} \cdots f_1 (v_1 \otimes v_{[1,\dots,k]} ) =  v_l \otimes  v_{[1,\dots,k]},
\end{align*}
 we have
\begin{align*}
 w= \Rnorm_{k,1}((-q)^{l-1}z)(v_l \otimes v_{[1,\dots,k]})=
 v_{[1,\dots,k]} \otimes v_l.
\end{align*}
It follows that
\begin{align} \label{eq:recursive akl}
a_{k,l}(z) = a_{k,l-1}(-q^{-1}z) \ a_{k,1}((-q)^{l-1}z) &\quad
(2 \leq l \leq k \leq n-2).
\end{align}
Using \eqref{eq:a1k}, \eqref{eq:recursive akl} and induction on $l$, we have
$$ a_{k,l} (z)\equiv\dfrac{[k-l][2n+k+l-2][2n-k-l-2][4n-k+l-4]}
    {[k+l][2n+k-l-2][2n-k+l-2][4n-k-l-4]}$$
   for $2 \le \l \le k \le n-2$.
    Because $a_{k,l}(z) = a_{l,k}(z)$, we obtain the desired result.
\end{proof}

For $1 \leq k, l \leq n-2$, set
\begin{align}
  D_{k,l}(z) = \prod_{s=1}^{\min(k,l)} (z-(-q)^{|k-l|+2s})(z-(-q)^{2n-k-l-2+2s}).
\end{align}
Then $D_{k,l}(z)=D_{l,k}(z)$ and we have
\eqn&&\ba{rl}
  a_{k,l}(z) a_{k,l}((-q)^{-2n+2}z) &\equiv
  \dfrac{(z-(-q)^{k+l})(z-(-q)^{2n-|k-l|-2})}%
  {(z-(-q)^{|k-l|})(z-(-q)^{2n-k-l-2})} \\[1ex]
  &\equiv \dfrac{D_{k,l}(z)}{D_{k,l}(q^{2n-2}z^{-1})}.\ea
  \eneqn
Hence \eqref{eq:ad} implies that
\eq
&&\dfrac{D_{i,j}(z)}{d_{i,j}(z)}\equiv
\dfrac{D_{i,j}(q^{2n-2}z^{-1})}{d_{i,j}(q^{2n-2}z^{-1})}.
\label{eq:Dd}
\eneq
By direct calculation, we also have
\begin{align} \label{eq:dkl dkl-1 k1}
  D_{k,l}(z) = D_{k,l-1}(-q^{-1}z)D_{k,1}((-q)^{l-1}z)
  & \quad (2 \leq l \le k \leq n-2).
\end{align}

Next, we will determine the denominators of the normalized $R$-matrices between the fundamental representations $\vpi{k}{}$ and $\vpi{l}{}$ for $1 \leq k,l \leq n-2$.
\begin{theorem}
For $1\leq k,l \leq n-2$, we have
\begin{align}
  d_{k,l}(z) =  D_{k,l}(z) =\prod_{s=1}^{\min (k,l)} (z-(-q)^{|k-l|+2s})(z-(-q)^{2n-k-l-2+2s}).
\end{align}
\end{theorem}
\begin{proof}
It is true for the case $k=l=1$ by \eqref{eq:deno11}.

Assume that $1 \leq k \leq n$ and $2 \leq l \leq n-2$.
By Corollary \ref{cor:k+l to k tensor l}, we have a surjective homomorphism
\begin{align*}
  p_{l-1,1} \col V(\varpi_{l-1})_{-q^{-1}} \otimes V_{(-q)^{l-1}}
  \epito V(\varpi_{l}).
  \end{align*}
Hence by Lemma \ref{lem:dvw avw}, we obtain
\begin{align}
\dfrac{d_{k,l-1}(-q^{-1}z)d_{k,1}((-q)^{l-1}z) a_{k,l}(z)}%
{d_{k,l}(z) a_{k,l-1}(-q^{-1}z) a_{k,1}((-q)^{l-1}z)} \in \cor[z, z^{-1}]
& \quad (1 \leq k \leq n, \ 2 \leq l \leq n-2).
\label{eq:dk1}
\end{align}
If $2 \le l \le k \le n-2$, then by \eqref{eq:recursive akl}, we have
\begin{align} \label{eq:dkl divides dkl-1 dk1}
  \dfrac{d_{k,l-1}(-q^{-1}z)d_{k,1}((-q)^{l-1}z)}{d_{k,l}(z) } \in \cor[z, z^{-1}] & \quad (2 \leq l \leq  k \leq n-2).
\end{align}
By \eqref{eq:a1k}, we have
$$\dfrac{a_{1,l}(z)}{a_{1,l-1}(-q^{-1}z) a_{1,1}((-q)^{l-1}z)} \equiv
\dfrac{(z-(-q)^{1-l})}{(z-(-q)^{3-l})}$$
for $2  \leq l \leq n-2$.
Hence setting $k=1$ in \eqref{eq:dk1} and then replacing $l$ with $k$,
we obtain
\eq&&
\ba{l}
 \dfrac{d_{1,k-1}(-q^{-1}z)D_{1,1}((-q)^{k-1}z)}{d_{1,k}(z)}
  \dfrac{(z-(-q)^{1-k})}{(z-(-q)^{3-k})} \\[1ex]
  \hs{8ex}\equiv \dfrac{d_{1,k-1}(-q^{-1}z)(z-(-q)^{2n-k-1})
  (z-(-q)^{1-k})}{d_{1,k}(z)}
    \in \cor[z, z^{-1}]\\[1ex]
    \hs{55ex}\text{for $2 \leq k\leq n-2$.}\ea
    \label{eq:d1l divides blabla}
\eneq

On the other hand, from the surjective homomorphism
\begin{align*}
V(\varpi_k)_{-q^{-1}} \otimes V(\varpi_1)_{(-q)^{2n-2-k}} \epito V(\varpi_{k-1}),
  \end{align*}
  we obtain
\begin{align*}
\dfrac{d_{1,l}((-q)^{k+2-2n}z)d_{k,l}(-qz) a_{k-1,l}(z)}{d_{k-1,l}(z) a_{k,l}(-qz) a_{1,l}((-q)^{k+2-2n}z)} \in \cor[z, z^{-1}] & \quad (2 \leq k \leq n-2, \  1 \leq l \leq n).
\end{align*}
Since
\begin{align*}
\dfrac{a_{k-1,l}(z)}
 {a_{k,l}(-qz) a_{1,l}((-q)^{k+2-2n}z)}   \equiv
\begin{cases}
  \dfrac{z-(-q)^{2n-k-l-3}}{z-(-q)^{2n-k-l-1}}
  & \text{if}  \ 1 \leq l < k \leq n-2, \\[2ex]
 \dfrac{( z-(-q)^{2n-k-l-3}) }{( z-(-q)^{2n-k-l-1}) }
 \dfrac{( z-(-q)^{-1} )}{( z-(-q) )}
  & \text{if} \ 2 \leq l = k \leq n-2,
 \end{cases}
\end{align*}
 we have
\begin{align}
&\dfrac{d_{1,l}((-q)^{k+2-2n}z)d_{k,l}(-qz) }{d_{k-1,l}(z)}
\dfrac{(z-(-q)^{2n-k-l-3})}{(z-(-q)^{2n-k-l-1})}
 \in \cor[z, z^{-1}] && \text{if $1 \leq l  < k \leq n-2$,}
  \label{eq:d1l dkl dk-1l l <k}\end{align}
and
\eq
&&\ba{l}\dfrac{d_{1,l}((-q )^{k+2-2n}z)d_{k,l}(-qz) }{d_{k-1,l}(z)}
 \dfrac{(z-(-q)^{2n-k-l-3})}{(z-(-q)^{2n-k-l-1})}
  \dfrac{(z-(-q)^{-1})}{(z-(-q))}
  \in \cor[z, z^{-1}]\\[1.8ex]
  \hs{50ex}\text{if $2 \leq l = k \leq n-2$.}\ea
    \label{eq:d1l dkl dk-1l l =k}
\eneq
Setting $l=1$ in \eqref{eq:d1l dkl dk-1l l <k}, we obtain
\begin{align} \label{eq:dk1 is diveded by blabla}
 \dfrac{D_{1,1}((-q)^{k+1-2n}z)d_{k,1}(z) }{d_{k-1,1}((-q)^{-1}z)} \dfrac{(z-(-q)^{2n-k-3})}{(z-(-q)^{2n-k-1})}
\quad \in \cor[z, z^{-1}] &\quad \text{for $2 \leq k \leq n-2$.}
\end{align}
\medskip

Now we will show $$d_{1,k}(z)=D_{1,k}(z)=(z-(-q)^{k+1})(z-(-q)^{2n-k-1})$$
 for $2 \leq k \leq n-2$.
By induction on $k$ in \eqref{eq:d1l divides blabla}, we have
\begin{align}\label{eq:d1k}
 &\dfrac{D_{1,k-1}(-q^{-1}z)(z-(-q)^{2n-k-1})(z-(-q)^{1-k})}{d_{1,k}(z)} \\
 =&(-q)^{-1}\dfrac{(z-(-q)^{k+1})(z-(-q)^{2n-k+1})
 (z-(-q)^{2n-k-1})(z-(-q)^{1-k})}{d_{1,k}(z)}
     \quad \in \cor[z, z^{-1}].\nonumber
\end{align}
for  $2 \leq k \leq n-2$.
Note that $z=(-q)^{1-k}$ is not a zero of $d_{1,k}(z)$
 since $(-q)^{1-k}\not\in \C[[q]]$.

 We shall show that $z=(-q)^{2n-k+1}$ is not a zero of $d_{1,k}(z)$ either.
By \eqref{eq:d1k}, $z=(-q)^{k-3}$ is not a zero of $d_{1,k}(z)$.
Hence $z=(-q)^{k-3}$ is not a pole of $D_{1,k}(z)/d_{1,k}(z)$, and therefore
 $z=(-q)^{2n-k+1}$ is not a pole of $D_{1,k}((-q)^{2n-2}z^{-1})/
d_{1,k}((-q)^{2n-2}z^{-1})$. Now \eqref{eq:Dd} implies that
$z=(-q)^{2n-k+1}$ is not a pole of $D_{1,k}(z)/d_{1,k}(z)$, which implies that
 $z=(-q)^{2n-k+1}$ is not a zero of $d_{1,k}(z)$.

Then, from \eqref{eq:d1k} we conclude that
\eq
\label{eq:dl1 divides}
&&\hs{2ex}\dfrac{D_{1,k}(z)}{d_{1,k}(z)}=\dfrac{(z-(-q)^{k+1})(z-(-q)^{2n-k-1})}{d_{1,k}(z)}
    \in \cor[z, z^{-1}]\ (2 \leq k \leq n-2).
\eneq
On the other hand, by induction on $k$ in \eqref{eq:dk1 is diveded by blabla},
we obtain
\begin{align}
  &\dfrac{d_{1,k}(z)}{(z-(-q)^{k+1})(z-(-q)^{2n-k-1})}
  \quad \in \cor[z, z^{-1}]  &&\quad \text{if} \ k \neq n-2,\\
  &\dfrac{d_{1,k}(z)}{z-(-q)^{2n-k-1}} \quad  \in \cor[z, z^{-1}] &&\quad \text{if} \ k = n-2.
\end{align}
By Lemma \ref{le:z-q^n-1}, $d_{1,k}(z)=0$ at $z=(-q)^{ k+1}$, we have
\eqn&&
\dfrac{d_{1,k}(z)}{D_{1,k}(z)}=\dfrac{d_{1,k}(z)}{(z-(-q)^{k+1})(z-(-q)^{2n-k-1})}
\in \cor[z, z^{-1}]\  (2\leq k \leq n-2).
\eneqn
Combining \eqref{eq:dl1 divides}, we conclude that
\begin{align}
d_{1,k}(z)= (z-(-q)^{k+1})(z-(-q)^{2n-k-1})=D_{1,k}(z) &\quad  (2 \leq k \leq n-2).
\end{align}
\medskip

Now we will show that  $d_{k,l}(z)=D_{k,l}(z)$ for $2 \le k, l \le n-2$
by induction on $k+l$.
Assume that $2 \leq l \leq k \leq n-2$.
By induction on $k+l$ in \eqref{eq:dkl divides dkl-1 dk1} and the relation
\eqref{eq:dkl dkl-1 k1}, we have
\eq&&\ba{l}
  \dfrac{d_{k,l-1}(-q^{-1}z)d_{k,1}((-q)^{l-1}z)}{d_{k,l}(z) }\\[1ex]
 \hs{10ex}=\dfrac{D_{k,l-1}(-q^{-1}z)D_{k,1}((-q)^{l-1}z) }{d_{k,l}(z) }
  \\
  \hs{20ex}=\dfrac{D_{k,l}(z)}{d_{k,l}(z)} \
   \in \cor[z, z^{-1}] \quad (2 \leq l \leq  k \leq n-2).\ea
\eneq
Define $\psi_{k,l}(z) \in \cor[z]$
by $D_{k,l}(z) = d_{k,l}(z) \psi_{k,l}(z)$.
We will show that $\psi_{k,l}(z)=1$ for $2 \leq l \leq k \leq n-2$.
Since
\begin{align*}
  &\dfrac{D_{1,l}((-q)^{k+2-2n}z)D_{k,l}((-q)z) }{D_{k-1,l}(z)}
  \dfrac{ (z-(-q)^{2n-k-l-3} ) }{( z-(-q)^{2n-k-l-1}) }  \\
  &\hs{10ex}=\begin{cases}
(z-(-q)^{4n-k-l-3})(z-(-1)^{2n-k-l-3}) & \text{if $l < k$,}\\[1.5ex]
 (z-(-q)^{4n-k-l-3})(z-(-1)^{2n-k-l-3}) (z-(-q)^{2n-1})(z+q)
 & \text{if $k=l$,}
   \end{cases}
\end{align*}
the relations \eqref{eq:d1l dkl dk-1l l <k},
\eqref{eq:d1l dkl dk-1l l =k} and induction on $k+l$ imply that
we have
\begin{align}
&\dfrac{(z-(-q)^{4n-k-l-3})(z-(-q)^{2n-k-l-3})}{\psi_{k,l}(-qz)} \ \in  \cor[z, z^{-1}]& \text{if} \ l < k,\\
&\dfrac{(z-(-q)^{4n-k-l-3})(z-(-q)^{2n-k-l-3})(z-(-q)^{2n-1})(z+q^{-1})}{\psi_{k,l}(-qz)}
\ \in  \cor[z, z^{-1}]
 & \text{if} \ k=l.
\end{align}
Since $\psi_{k,l}(-qz)$ divides $D_{k,l}(-qz)$,
we conclude that
\begin{align}
&\psi_{k,l}(z)=1 &&\quad \text{if } \ k+l < n-1, \\
&\dfrac{(z-(-q)^{2n-k-l-2})}{\psi_{k,l}(z)} \ \in  \cor[z, z^{-1}]
&&\quad \text{if} \ k+l \geq n-1.\label{eq:psi}
\end{align}

Hence it is enough to show that
$z=(-q)^{2n-k-l-2}$ is not a zero of $\psi_{k,l}(z)$ when $k+l\ge n-1$.
The relation \eqref{eq:Dd} implies that
$\psi_{k,l}(q^{2n-2}z^{-1})\equiv\psi_{k,l}(z)$.
Hence it is enough to show that
$z=(-q)^{k+l}$ is not a zero of $\psi_{k,l}(z)$ when $k+l\ge n-1$.
It is a consequence of \eqref{eq:psi} if $k+l>n-1$.

Assume that  $k+l=n-1$. Then Lemma \ref{le:z-q^n-1}
implies that $z=(-q)^{k+l}$ is a zero of $d_{k,l}(z)$.
Since $z=(-q)^{k+l}$ is a simple zero of $D_{k,l}(z)$,
we conclude that $z=(-q)^{k+l}$ is not a zero of $\psi_{k,l}(z)$
when $k+l=n-1$.
\end{proof}

\end{document}